\documentclass[12pt,a4paper]{amsart}
\usepackage{esint}
\usepackage{amssymb}
\usepackage{MnSymbol}
\usepackage{graphicx}

\usepackage[
hypertexnames=false, colorlinks, 
linkcolor=black,citecolor=black,urlcolor=black, linktocpage=true
]%
{hyperref} 
\hypersetup{bookmarksdepth=3}

	\normalbaselines						  %

\newtheorem{thm}{Theorem}[section]
\newtheorem{lm}[thm]{Lemma}
\newtheorem{cor}[thm]{Corollary}
\newtheorem{prop}[thm]{Proposition}

\theoremstyle{definition}

\newtheorem*{df*}{Definition}

\theoremstyle{remark}

\newtheorem*{rem*}{Remark}

\numberwithin{equation}{section}

\newcommand{\ci}[1]{_{ {}_{\scriptstyle #1}}}
\newcommand{\ti}[1]{_{\scriptstyle \text{\rm #1}}}

\newcommand{\mathd}{\mathrm{d}}

\newcommand{\cI}{\mathcal{I}}
\newcommand{\cD}{\mathcal{D}}
\newcommand{\cC}{\mathcal{C}}

\newcommand{\cF}{\mathcal{F}}

\newcommand{\cU}{\mathcal{U}}

\newcommand{\e}{\varepsilon}

\newcommand{\D}{\mathbb{D}}

\newcommand{\R}{\mathbb{R}}
\newcommand{\Z}{\mathbb{Z}}

\newcommand{\bP}{\mathbb{P}}

\newcommand{\Om}{\Omega}

\newcommand{\1}{\mathbf{1}}

\newcommand{\cz}{Calder\'{o}n--Zygmund\ }

\newcommand{\La}{\langle }
\newcommand{\Ra}{\rangle }

\newcommand{\bs}[1]{\boldsymbol{#1}}
\newcommand{\bfD}{\bs{\mathcal{D}}}

\newcommand{\bsigma}{\bs{\Sigma}}

\newcommand{\ddoto}{\"{o}}
\newcommand{\ddotu}{\"{u}}

\newcommand{\fdot}{\,\cdot\,}

\DeclareMathOperator*{\esssup}{ess\,sup}

\newcommand{\BMO}{\ensuremath\text{BMO}}
\newcommand{\BMOA}{\ensuremath\text{BMOA}}
\newcommand{\BMOprodD}{\ensuremath\text{BMO}\ci{\text{prod},\bfD}}
\newcommand{\sh}{\text{sh}}

\def\cyr{\fontencoding{OT2}\fontfamily{wncyr}\selectfont}
\DeclareTextFontCommand{\textcyr}{\cyr}

\newenvironment{entry}
{\begin{list}{X}%
		{%
			\setlength{\labelwidth}{55pt}%
			\setlength{\leftmargin}{\labelwidth}
			\addtolength{\leftmargin}{\labelsep}%
			\setlength{\itemsep}{.4pc}
	}%
}%
{\end{list}}

\newcounter{vremennyj}

\begin{document}

\title[Dyadic product BMO in the Bloom setting]{Dyadic product BMO in the Bloom setting}
\author{Spyridon Kakaroumpas}
\address{Institut f\"{u}r Mathematik\\ Universit\"{a}t W\"{u}rzburg\\97074 W\ddotu rzburg\\ Germany}
\email{spyridon.kakaroumpas@mathematik.uni-wuerzburg.de}
\author{Od\'i Soler i Gibert}
\address{Institut f\"{u}r Mathematik\\ Universit\"{a}t W\"{u}rzburg\\97074 W\ddotu rzburg\\ Germany}
\email{odi.solerigibert@mathematik.uni-wuerzburg.de}
\thanks{Spyridon Kakaroumpas is supported by the  Alexander von Humboldt Stiftung.}
\thanks{Od\'i Soler i Gibert is supported by the ERC project CHRiSHarMa no. DLV-682402.}
\date{}

\begin{abstract}
\'O.~Blasco and S.~Pott showed that the supremum of operator norms over $L^2$ of all bicommutators (with the same symbol) of one-parameter Haar multipliers dominates the biparameter dyadic product BMO norm of the symbol itself. In the present work we extend this result to the Bloom setting, and to any exponent $1<p<\infty$. The main tool is a new characterization in terms of paraproducts and two-weight John--Nirenberg inequalities for dyadic product BMO in the Bloom setting.
We also extend our results to the whole scale of indexed spaces between little bmo and product BMO in the general multiparameter setting, with the appropriate iterated commutator in each case.
\end{abstract}

\maketitle
\setcounter{tocdepth}{1}

\tableofcontents
\setcounter{tocdepth}{2}

\section*{Notation}

\begin{entry}
\item[$\1\ci{E}$] characteristic function of a set $E$;

\item[$dx$] integration with respect to Lebesgue measure; 

\item[$|E|$] $d$-dimensional Lebesgue measure of a measurable set $E\subseteq\R^d$;

\item[$\La f\Ra\ci{E}$] average with respect to Lebesgue measure, $\La f\Ra\ci{E}:=\frac{1}{|E|}\int_{E}f(x)dx$;

\item[$L^{\infty}\ti{c}$] space of compactly supported $L^{\infty}$ functions;

\item[$L^{p}(w)$] weighted Lebesgue space, $\|f\|\ci{L^p(w)}^p := \int_{\R^d}|f(x)|^p w(x) dx$; 

\item[$\La f,g\Ra$] usual $L^2$-pairing, $\La f,g\Ra := \int f(x) \overline{g(x)} dx$;

\item[$w(E)$] Lebesgue integral of a weight $w$ over a set $E$, $w(E):=\int_{E}w(x)dx$;

\item[$p'$] H\"{o}lder conjugate exponent to $p$, $1/p+1/p'=1$; 

\item[$\cD$] family of all dyadic intervals in $\R$;

\item[$I_{-},\,I_{+}$] respectively, left and right half of an interval $I\in\cD$;

\item[$\bfD$] family of all dyadic rectangles in the product space $\R\times\R$;

\item[$\text{sh}(\cU)$] ``shadow" of a family $\cU$ of dyadic rectangles, $\text{sh}(\cU):=\bigcup_{R\in\cU}R$;

\item[$h^{(0)}\ci{I},~h^{(1)}\ci{I}$] $L^{2}$-normalized \emph{cancellative} and \emph{non-cancellative} respectively Haar functions for an interval $I\in\cD$, $h^{(0)}\ci{I}:=\frac{\1\ci{I_{+}}-\1\ci{I_{-}}}{\sqrt{|I|}}$, $h^{(1)}\ci{I}:=\frac{\1\ci{I}}{\sqrt{|I|}}$; for simplicity we denote $h\ci{I}:=h^{(0)}\ci{I}$;

\item[$b\ci{I}$] usual Haar coefficient of a function $b\in L^1\ti{loc}(\R)$, $b\ci{I}:=\La b,h\ci{I}\Ra$, $I\in\cD$;

\item[$h\ci{R}^{(\e_1\e_2)}$] any of the four $L^2$-normalized Haar functions for a rectangle $R\in\bfD$, $h\ci{R}^{(\e_1\e_2)}:=h\ci{I}^{(\e_1)}\otimes h\ci{J}^{(\e_2)}$, where $R=I\times J$ and $\e_1,\e_2\in\lbrace0,1\rbrace$; for simplicity we denote $h\ci{R}:=h\ci{R}^{(00)}$;

\item[$b\ci{R}$] (00) Haar coefficient of a function $b\in L^1\ti{loc}(\R^2)$, $b\ci{R}:=\La b,h\ci{R}\Ra$, $R\in\bfD$;

\item[$T^{\ast}$] formal $L^2$-adjoint operator to the operator $T$, $\La Tf,g\Ra=\La f,T^{\ast}g\Ra$.
\end{entry}

The notation $x\lesssim\ci{a,b,\ldots}  y$ means $x\leq Cy$ with a constant $0<C<\infty$ depending \emph{only} on the quantities $a, b, \ldots$; the notation $x\gtrsim\ci{a,b,\ldots} y$ means $y\lesssim\ci{a,b,\ldots} x$.  We use  $x\sim\ci{a,b,\ldots} y$ if \emph{both} $x\lesssim\ci{a,b,\ldots} y$ and $x\gtrsim\ci{a,b,\ldots} y$ hold. Sometimes we might omit some of these quantities $a,b,\ldots$ from the notation. The context will always make clear when this happens.

\section{Introduction and main results}
In 1957, Z.~Nehari \cite{nehari} showed that Hankel operators $\mathrm{H}_b$ are bounded from the Hardy space $\mathcal{H}^2(\partial\D)$ into itself if and only if the symbol $b$ belongs to the space of analytic functions with bounded mean oscillation, or simply $b \in \BMOA.$
In fact, Nehari \cite{nehari} shows an equivalence between the norm of the operator $\mathrm{H}_b$ and the $\BMOA$ norm of the symbol $b.$
His proof relies on the fact that a function $f$ in the Hardy space $\mathcal{H}^1(\partial\D)$ can be factored as $f = g_1 g_2,$ where both $g_1,g_2 \in \mathcal{H}^2(\partial\D).$
Note here that the Hankel operator $\mathrm{H}_b$ is essentially equivalent to the commutator $[H,b],$ where $H$ denotes the Hilbert transform on $\partial\D$ and (abusing notation) $b$ stands for multiplication by this function.
This allows one to consider not only the real variable version of Nehari's result, but also analogues in $\R^d$ for any $d$ by studying commutators of the form $[R,b],$ where $R$ denotes one of the Riesz transforms in $\R^d.$
Nonetheless, observe that the factorization for $\mathcal{H}^1(\partial\D)$ has no counterpart in this setting (see also \cite{coifman-rochberg-weiss}).

In this direction, R.~R.~Coifman, R.~Rochberg and G.~Weiss \cite{coifman-rochberg-weiss} proved in 1976 their celebrated commutator theorem.
Namely, for a function $b$ with bounded mean oscillation on $\R^d,$ denoted $b \in \BMO(\R^d),$ they show an equivalence between the $\BMO$ norm of $b$ and the sum of the norms of the commutators $[R^{(j)},b],~j=1,\ldots,d$ with $R^{(j)}$ denoting the $j$-th Riesz transform in $\R^d,$ as operators from $L^p$ into itself, for any $1 < p < \infty.$
They also show that for a \cz operator $T,$ the norm of $[T,b]$ as an operator from $L^p$ into itself, $1 < p < \infty,$ is bounded above by $\Vert b \Vert\ci{\BMO}.$
The argument used by Coifman--Rochberg--Weiss \cite{coifman-rochberg-weiss} to show the lower bound for the sum of commutators with the Riesz transforms is based on a decomposition of the identity as a linear combination of products of Riesz transforms using spherical harmonics in $\R^d.$
This allows one to bound the oscillation of a $\BMO$ function by the sum of the norms of $[R^{(j)},b]$, $j=1,\ldots,d$. 

In a different direction, S.~Bloom \cite{bloom} proved in 1985 an analogue of Nehari's result for weighted spaces.
Bloom \cite{bloom} showed a norm equivalence between a certain weighted $\BMO$ space and the operator norm of the commutator $[H,b]$, where $H$ is again the Hilbert transform.
To be more precise, for a function $b \in L^1\ti{loc}(\R^d),$ and for a weight $\nu$ on $\R^d$ (that is a locally integrable, a.e. positive function on $\R^d$), define the one-weight $\BMO$ norm
\begin{equation*}
  \Vert b \Vert\ci{\BMO(\nu)} := \sup_{Q} \frac{1}{\nu(Q)} \int_{Q} |b(x) - \La b \Ra\ci Q|\, dx,
\end{equation*}
where the supremum ranges over all cubes $Q\subseteq \R^d$ and $\La b \Ra\ci Q = \frac{1}{|Q|} \int_Q b(x)\, dx$ is the unweighted average of $b$ on $Q.$
This weighted BMO space had already been investigated by B.~Muckenhoupt and R.~L.~Wheeden \cite{muckenhoupt-wheeden}.
Bloom \cite{bloom} showed that for any $1 < p < \infty,$ for any $A_p$ weights $\mu$ and $\lambda$ on $\R$ (see Section \ref{sec:Background} for the definition of $A_p$ weights), and for the weight $\nu:=\mu^{1/p}\lambda^{-1/p}$, which is easily seen to be an $A_2$ weight, one has the equivalence
\begin{equation}
\label{Bloom result}
  \Vert [H,b] \Vert\ci{L^p(\mu) \rightarrow L^p(\lambda)} \sim \Vert b \Vert\ci{\BMO(\nu)},
\end{equation}
where $H$ denotes the Hilbert transform and the implied constants depend only on $p$ and the $A_p$ characteristics of $\mu$ and $\lambda.$ Bloom's \cite{bloom} proof of the estimate $\Vert [H,b] \Vert\ci{L^p(\mu) \rightarrow L^p(\lambda)} \gtrsim \Vert b \Vert\ci{\BMO(\nu)}$ uses a different argument to that of Coifman--Rochberg--Weiss, and is based on a careful analysis of the set where $|b - \La b \Ra\ci{I}|$ is not too large with respect to the average oscillation on the interval $I,$ and the set where the former quantity is large using the adjoint commutator.

Much more recently, I.~Holmes, M.~T.~Lacey and B.~D.~Wick \cite{holmes-lacey-wick} considerably extended Bloom's result, proving that for any \cz operator $T$ on $\R^d$ and for any $A_p$ weights $\mu,\lambda$ on $\R^d$, $1<p<\infty$, there holds
\begin{equation}
\label{hlw-lower}
\Vert [T,b]\Vert\ci{L^p(\mu)\rightarrow L^{p}(\lambda)}\lesssim\ci{T,d,p}\Vert b\Vert\ci{\text{BMO}(\nu)},
\end{equation}
and that for the Riesz transforms $R^{(1)},\ldots,R^{(d)}$ on $\R^d$ there holds in addition
\begin{equation}
\label{hlw-upper}
\sum_{i=1}^{d}\Vert [R^{(i)},b]\Vert\ci{L^p(\mu)\rightarrow L^{p}(\lambda)}\gtrsim_{p,d}\Vert b\Vert\ci{\text{BMO}(\nu)},
\end{equation}
where in both \eqref{hlw-lower} and \eqref{hlw-upper}, $\nu:=\mu^{1/p}\lambda^{-1/p}$, and all implied constants depend on $p$ and the $A_p$ characteristics of $\mu,\lambda$ as well. Holmes--Lacey--Wick \cite{holmes-lacey-wick} proved and used in an essential way several new characterizations of the weighted BMO space $\text{BMO}(\nu)$ in the form of two-weight John--Nirenberg inequalities. More precisely, Muckenhoupt--Wheeden \cite{muckenhoupt-wheeden} had already showed that if $\nu$ is an $A_2$ weight on $\R^d$, then one has the equivalence
\begin{equation*}
  \Vert b \Vert\ci{\BMO(\nu)} \sim \sup_{Q} \left(\frac{1}{\nu(Q)} \int_{Q} |b(x) - \La b \Ra\ci Q|^2\,\nu^{-1}(x) dx\right)^{1/2},
\end{equation*}
where the implied constants depend only on $d$ and the $A_2$ characteristic of $\nu$. Holmes--Lacey--Wick \cite{holmes-lacey-wick} complemented this result, by showing that if $\mu,\lambda$ are $A_p$ weights on $\R^d$, $1<p<\infty$, and $\nu:=\mu^{1/p}\lambda^{-1/p}$, then one has the equivalences
\begin{equation*}
  \Vert b \Vert\ci{\BMO(\nu)} \sim \sup_{Q} \left(\frac{1}{\mu(Q)} \int_{Q} |b(x) - \La b \Ra\ci Q|^{p}\,\lambda(x) dx\right)^{1/p}
\end{equation*}
and
\begin{equation*}
\Vert b \Vert\ci{\BMO(\nu)} \sim \sup_Q \left(\frac{1}{\lambda'(Q)} \int_{Q} |b(x) - \La b \Ra\ci Q|^{p'}\,\mu'(x) dx\right)^{1/p'}
\end{equation*}
where $\mu':=\mu^{-1/(p-1)}$, $\lambda':=\lambda^{-1/(p-1)}$, and all implied constants depend only on $d,p$ and the $A_p$ characteristics of $\mu$ and $\lambda$. Their proof of these equivalences employed a duality result between \emph{dyadic} $\text{BMO}(\nu)$ and a certain dyadic weighted $H^1$ space that they established in the same work, as well as characterizations of two-weight BMO spaces in terms of two-weight boundedness of certain paraproducts. It should be noted that the results of \cite{holmes-lacey-wick} were very recently extended to the matrix-valued setting by J.~Isralowitz, S.~Pott and S.~Treil \cite{treil et al}. In fact, the authors of \cite{treil et al} proved there several results for the case of completely arbitrary (not necessarily $A_p$) matrix-valued weights, that are new even if one specializes to the fully scalar setting.

All results mentioned above concern one-parameter spaces. On the other hand, multiparameter (unweighted) BMO spaces were investigated extensively in the seminal papers by \mbox{S.-Y.~A.~Chang \cite{chang}} and R.~ Fefferman \cite{fefferman} in the late 1970s. These works concern mainly the biparameter product BMO space $\text{BMO}(\R\times\R)$ on the product space $\R\times\R$, defined by
\begin{equation}
\label{chang-fefferman BMO}
\Vert b\Vert\ci{\text{BMO}(\R\times\R)}:=\sup_{\Omega} \bigg(\frac{1}{|\Omega|} \sum_{\substack{R\in\bfD\\R\subseteq\Omega}}|\La b,w\ci{R}\Ra|^2 \bigg)^{1/2},
\end{equation}
where the supremum ranges over all non-empty open sets $\Omega$ of finite measure, $\bfD$ is the family of all dyadic rectangles in the product space $\R\times\R$ (with sides parallel to the coordinate axes), and $(w\ci{R})\ci{R\in\bfD}$ is some (regular enough) wavelet system adapted to dyadic rectangles. Note that an analogous definition of multiparameter product BMO can be given in any product space $\R^{\vec{d}}:=\R^{d_1}\times\dots\times\R^{d_t}$. Works \cite{chang} and \cite{fefferman} provide equivalent descriptions of biparameter product BMO spaces in terms of Carleson measures, extending the one-parameter classical ones. It is important to note that in definition \eqref{chang-fefferman BMO}, one \emph{cannot} restrict the supremum to rectangles. This follows from a famous counterexample due to L.~ Carleson \cite{carleson}, recounted in Fefferman's article \cite{fefferman} (see also \cite{blasco-pott-cc} or \cite{tao}).

The first breakthrough in the study of the relation between norms of commutators and the BMO norm of their symbol in the multiparameter setting was achieved by S.~H.~Ferguson and C.~Sadosky  \cite{ferguson-sadosky} . The authors of \cite{ferguson-sadosky} proved there that
\begin{equation*}
\Vert [H\otimes H,b]\Vert\ci{L^2(\R^2)\rightarrow L^2(\R^2)}\sim \Vert b\Vert\ci{\text{bmo}(\R\times\R)},
\end{equation*}
where $H$ is the Hilbert transform, $H\otimes H$ is a tensor product of Hilbert transforms (each acting on one of the two variables), and $\text{bmo}(\R\times\R)$ is the so-called \emph{little bmo} space,
\begin{equation*}
\Vert b\Vert\ci{\text{bmo}(\R\times\R)}:=\sup_{R}\frac{1}{|R|}\int_{R}|b(x)-\La b\Ra\ci{R}|dx,
\end{equation*}
where the supremum is taken over all rectangles $R$ in $\R\times\R$ (with sides parallel to the coordinate axes).
In \cite{ferguson-sadosky}, it is also established an upper bound for iterated commutators.
Namely, if $H_1$ and $H_2$ denote, respectively, the Hilbert transforms acting on the first and second variable, then one has the upper bound 
\begin{equation}
  \label{eq:ferguson-sadosky-upper}
  \Vert [H_1,[H_2,b]]\Vert\ci{L^2(\R^2)\rightarrow L^2(\R^2)}
  \lesssim \Vert b\Vert\ci{\text{BMO}(\R\times\R)}.
\end{equation}
Later, L.~Dalenc and Y.~Ou \cite{dalenc-ou} proved that if $T_1, T_2$ are (usual one-parameter) \cz operators acting on the first and second respectively variables of $\R^{\vec{d}}:=\R^{d_1}\times\R^{d_2}$ respectively, then
\begin{equation}
\label{dalenc-ou-upper}
\Vert [T_1,[T_2,b]]\Vert\ci{L^2(\R^{\vec{d}})\rightarrow L^2(\R^{\vec{d}})}\lesssim\ci{T_1,T_2,\vec{d}}\Vert b\Vert\ci{\text{BMO}(\R^{\vec{d}})}
\end{equation}
(in fact, Dalenc--Ou \cite{dalenc-ou} established an analogous result in any number of parameters).

The result of Ferguson--Sadosky \cite{ferguson-sadosky} was generalized and also extended to the weighted setting by I.~Holmes, S.~Petermichl and B.~D.~Wick \cite{holmes-petermichl-wick}. There, the authors proved that if $T$ is any biparameter \cz operator (aka Journ\'e operator) on $\R^{\vec{d}}:=\R^{d_1}\times\R^{d_2}$, then
\begin{equation}
\label{hpw-upper}
\Vert[T,b]\Vert\ci{L^{p}(\mu)\rightarrow L^{p}(\lambda)}\lesssim_{T,\vec{d},p} \Vert b\Vert\ci{\text{bmo}(\nu,\R^{\vec{d}})},
\end{equation}
and that if moreover $R^{(1)},\ldots,R^{(d)}$ are the Riesz tranforms on $\R^d$, then
\begin{equation}
\label{hpw-lower}
\sum_{k,l=1}^{d}\Vert[R^{(k)}\otimes R^{(l)},b]\Vert\ci{L^{p}(\mu)\rightarrow L^{p}(\lambda)}\gtrsim_{p,d} \Vert b\Vert\ci{\text{bmo}(\nu,\R^{d}\times\R^{d})},
\end{equation}
where in both results, $\mu$ and $\lambda$ are biparameter $A_p$ weights on $\R^{\vec{d}}$ (see Section \ref{sec:Background} for the definition), all implied constants depend on the biparameter $A_p$ characteristics of $\mu$, $\lambda$ as well, $\nu:=\mu^{1/p}\lambda^{-1/p}$, $1<p<\infty$, and
\begin{equation*}
\Vert b\Vert\ci{\text{bmo}(\nu,\R^{\vec{d}})}:=\sup_{R}\frac{1}{\nu(R)}\int_{R}|b(x)-\La b\Ra\ci{R}|dx,
\end{equation*}
where the supremum is again taken over all rectangles $R$ in $\R^{\vec{d}}$ (with sides parallel to the coordinate axes). Holmes--Petermichl--Wick \cite{holmes-petermichl-wick} proved and used in an essential way two-weight John--Nirenberg inequalities for the weighted space $\text{bmo}(\nu,\R^{\vec{d}})$ that are analogous to the ones in the one-parameter setting in \cite{holmes-lacey-wick}, in order to prove the lower bound \eqref{hpw-lower}. For the upper bound \eqref{hpw-upper}, Holmes--Petermichl--Wick \cite{holmes-petermichl-wick} defined and used the \emph{dyadic} product Bloom space $\BMOprodD(\nu,\R^{\vec{d}})$,
\begin{equation*}
\Vert b \Vert\ci{\BMOprodD(\nu,\R^{\vec{d}})} := \sup_\cU \left( \frac{1}{\nu(\sh(\cU))} \sum_{R\in\cU} |b_R|^2 \La \nu^{-1} \Ra_R \right)^{1/2},
\end{equation*}
where the supremum ranges over all non-empty collections $\cU$ of \emph{dyadic} rectangles in the biparameter product space $\R^{\vec{d}}$,
\begin{equation*}
 \sh(\cU) := \bigcup_{R\in\cU} R,
\end{equation*}
and $b\ci{R}:=\La b,h\ci{R}\Ra$, where $h\ci{R}$ is the cancellative in both-variables $L^2$-normalized Haar function over the dyadic rectangle $R$. In \cite{holmes-petermichl-wick} a duality result between $\BMOprodD(\nu,\R^{\vec{d}})$ and a weighted $H^1$ space is established, which is essential for the proof of the upper bound \eqref{hpw-upper}. The authors of \cite{holmes-petermichl-wick} also extended the upper BMO bound \eqref{dalenc-ou-upper} due to Dalenc--Ou  in \cite{dalenc-ou} to the case of multiparameter indexed unweighted BMO spaces.

Since then, weighted product BMO and multiparameter indexed weighted BMO upper bounds have been investigated and established in full generality by E.~Airta, K.~Li, H.~Martikainen, E.~Vuorinen \cite{airta}, \cite{martikainen et al}, \cite{li-martikainen-vuorinen1}, \cite{li-martikainen-vuorinen2}.

It is important to note that in all of the aforementioned works, the main tools for establishing upper bounds for norms of commutators with \cz operators in terms of the BMO norm of their symbol is the decomposition theorem of \cz operators in terms of Haar shifts and paraproducts that T.~Hyt\ddoto nen established and used to prove his $A_2$ theorem \cite{hytonena2}, as well as its extension to the multiparameter setting due to Martikainen \cite{martikainen}.

While upper bounds for norms of commutators in terms of multiparameter BMO norms of the symbol are by now well-understood, in the fully general two-weight setting, the picture for lower bounds remains incomplete even in the unweighted case.
The study of such lower bounds was addressed by S.~H.~Ferguson and M.~T.~Lacey \cite{ferguson-lacey}, who gave a converse to \eqref{eq:ferguson-sadosky-upper}, namely
\begin{equation}
  \label{eq:ferguson-lacey-lower}
  \Vert [H_1,[H_2,b]]\Vert\ci{L^2(\R^2)\rightarrow L^2(\R^2)}
  \gtrsim \Vert b\Vert\ci{\text{BMO}(\R\times\R)}.
\end{equation}
The proof of this fact was based on new and beautiful arguments on this matter.
More recently, Dalenc--Petermichl \cite{dalenc-petermichl} proved similar results for iterated commutators of Riesz transforms.
Moreover, Y.~Ou, S.~Petermichl and E.~Strouse \cite{ou-petermichl-strouse} extended the result in Dalenc--Petermichl \cite{dalenc-petermichl} to the case of multiparameter indexed BMO spaces, which are between little bmo and multiparameter Chang--Fefferman product BMO. 
Both works \cite{dalenc-petermichl} and \cite{ou-petermichl-strouse} rely on the result of Ferguson--Lacey \cite{ferguson-lacey}.

There are as well some lower bounds that do not rely on \eqref{eq:ferguson-lacey-lower}, like \cite{holmes-lacey-wick}, \cite{holmes-petermichl-wick}, \cite{treil et al}.
However, all these employ variants of the original argument by Coifman--Rochberg--Weiss \cite{coifman-rochberg-weiss}. 	Arguments of this type rely on the availability of explicit ``oscillatory'' expressions for BMO norms.
While such expressions are indeed available in the one-parameter setting, and also in the case of the little bmo space, they are not at all available in the case of product BMO (in the sense of Chang--Fefferman), making investigating such lower bounds significantly harder.

In another direction, \'O.~Blasco and S.~Pott \cite{blasco-pott} related dyadic biparameter product BMO norms to iterated commutators of Haar multipliers. More precisely, consider the set $\Sigma$ of all finitely supported maps $\sigma\colon \cD \rightarrow \{-1,0,1\},$ and for each $\sigma \in \Sigma$ consider its Haar multiplier $T_\sigma$ on $L^2(\R)$ (see Section \ref{sec:HaarMultipliers} for precise definitions). Furthermore, consider Haar multipliers $T^1_{\sigma_1}$ and $T^2_{\sigma_2}$ acting on $L^2(\R^2)$ separately on each variable. Blasco--Pott \cite{blasco-pott} show that
\begin{equation*}
\sup_{\sigma_1,\sigma_2\in\Sigma} \Vert [T^1_{\sigma_1},[T^2_{\sigma_2},b]] \Vert\ci{L^2(\R^2)\rightarrow L^2(\R^2)} \sim \Vert b \Vert\ci{\BMOprodD}.
\end{equation*}
It should be noted that the supremum over all signs enables Blasco--Pott \cite{blasco-pott} to eliminate error terms by taking average over all signs and then use orthogonality arguments in order to conclude their result.

The main goal of the present paper is to extend the aforementioned result by Blasco--Pott \cite{blasco-pott} to the weighted setting, and to the full range of exponents $1<p<\infty$. More precisely, we show the following.
\begin{thm}
\label{t: HaarmultBMO}
Let $1 < p < \infty.$
Consider a function $b \in L^1\ti{loc}(\R^2),$ dyadic biparameter $A_p$ weights $\mu,$ $\lambda$ and define $\nu:=\mu^{1/p}\lambda^{-1/p}$. Then
\begin{equation}
\label{eq:HaarMultBMOEquiv}
\sup_{\sigma_1,\sigma_2\in\Sigma}\Vert[T_{\sigma_1}^1,[T_{\sigma_2}^2,b]]\Vert\ci{L^p(\mu)\rightarrow L^p(\lambda)}\sim\Vert b\Vert\ci{\emph{BMO}\ci{\emph{prod},\bfD}(\nu)},
\end{equation}
where the implied constants depend only on $p,$ $[\mu]\ci{A_p,\bfD}$ and $[\lambda]\ci{A_p,\bfD}$.
\end{thm}
The proof of this theorem will follow similar steps to that of the result due to Blasco--Pott.
Namely, first we show that the supremum in the left-hand side of \eqref{eq:HaarMultBMOEquiv} is equivalent to the norm of a certain operator defined in terms of paraproducts.
Then we show the equivalence between the previous operator norm and the BMO norm of the symbol $b.$
Nonetheless, here we must use additional techniques to overcome the lack of orthogonality for $p \neq 2.$
In particular, we make use of a multiparameter extension of the classical Khintchine's inequality together with vector-valued estimates.
Another possibility that also allows one to circumvent this difficulty would be to use duality coupled with Hölder's inequality, together with vector-valued estimates.
Moreover, to be able to handle the weighted spaces appearing in Theorem \ref{t: HaarmultBMO}, we establish equivalent characterizations of dyadic product Bloom BMO in the spirit of \cite{holmes-lacey-wick}, \cite{holmes-petermichl-wick}. More precisely, fix $1<p<\infty$ and two dyadic biparameter $A_p$ weights $\mu$ and $\lambda$ on $\R^2.$
Given $b \in L^1\ti{loc}(\R^2)$, define the dyadic two-weight Bloom product $\BMO$ norm
\begin{equation*}
  \Vert b \Vert\ci{\BMOprodD(\mu,\lambda,p)} := \sup_\cU \frac{1}{(\mu(\sh(\cU)))^{1/p}} \Vert S\ci{\cU}(b)\Vert\ci{L^{p}(\lambda)},
\end{equation*}
where the supremum ranges again over all non-empty collections $\cU$ of dyadic rectangles in $\R\times\R$, and $S\ci{\cU}(b)$ is the biparameter dyadic square function of $b$ restricted to the collection $\cU$ (see Section \ref{sec:Background} for precise definitions). We show the following two-weight John--Nirenberg inequalities.

\begin{thm}
\label{t: equivBMO}
Let $1 < p < \infty.$
Consider dyadic biparameter $A_p$ weights $\mu,$ $\lambda$ and define $\nu:=\mu^{1/p}\lambda^{-1/p}$. Then
\begin{equation*}
\Vert b\Vert\ci{\emph{BMO}\ci{\emph{prod},\bfD}(\nu)}\sim\Vert b\Vert\ci{\emph{BMO}\ci{\emph{prod},\bfD}(\mu,\lambda,p)},
\end{equation*}
where the implied constants depend only on $p,$ $[\mu]\ci{A_p,\bfD}$ and $[\lambda]\ci{A_p,\bfD}$.
\end{thm}
While the formulation of Theorem \ref{t: equivBMO} reflects that of the two-weight John--Nirenberg inequalities in the one-parameter setting \cite{holmes-lacey-wick} and in the case of little bmo \cite{holmes-petermichl-wick}, its proof addresses several new difficulties not present in \cite{holmes-lacey-wick}, \cite{holmes-petermichl-wick}, and requires new ideas.
We split this proof in several steps.
For $1 < p \leq 2$ we show that the two-weight norm $\Vert b\Vert\ci{\text{BMO}\ci{\text{prod},\bfD}(\mu,\lambda,p)}$ is equivalent to a certain paraproduct norm, as an operator from $L^p(\mu)$ to $L^p(\lambda).$
Then, we prove the equivalence between the norm of this paraproduct and the one-weight norm $\Vert b\Vert\ci{\text{BMO}\ci{\text{prod},\bfD}(\nu)}$ for any $1 < p < \infty.$
Although the remaining equivalence, that is for $2 < p < \infty,$ is an immediate consequence of Hölder's inequality in the unweighted case, in our setting it requires the use of the so called (biparameter) Triebel--Lizorkin square function (see Section \ref{sec:Background} for the definition and basic properties).
In addition, it is essential for this step of the proof to make use of an equivalence between one-weight and unweighted product BMO due to E.~Airta, K.~Li, H.~Martikainen and E.~Vuorinen \cite{martikainen et al}.
In particular, the equivalence from \cite{martikainen et al} that we use corresponds to the particular case of our Theorem \ref{t: equivBMO} when $p > 2$ and $\mu = \lambda.$
These two ingredients are necessary to overcome the lack of a ``two-weight Hölder's inequality''.

Note that John--Nirenberg inequalities hold for weighted little bmo, for all $1 < p < \infty$ (see \cite{holmes-petermichl-wick}).
Moreover, it was already known that they also hold in unweighted product BMO as well, for all $1 < p < \infty$ (see \cite{tao} for a proof using atomic decompositions).
Finally, for rectangular BMO, John--Nirenberg inequalities do not hold even in the unweighted case, and for any $1<p<\infty$ (see \cite{blasco-pott}).

It should also be noted that all of our results hold for functions defined on any multiparameter product space $\R^{d_1}\times\dots\times\R^{d_t},$ with identical or similar proofs.
Although for simplicity we restrict the main part of this work to the case of functions on $\R\times\R,$ we also indicate how to modify the arguments for the general multiparameter setting when appropriate.
Moreover, we explain briefly how to extend our results to the whole scale of indexed spaces between little bmo and product BMO in the general multiparameter setting, with the appropriate iterated commutator in each case.

\textbf{Plan of the paper.} The article is structured as follows. In Section \ref{sec:Background} the reader can find the notations and definitions that will be used in the rest of the paper. In Section \ref{sec:WeightedEquivalences} we prove the two-weight John--Nirenberg inequalities of Theorem \ref{t: equivBMO}. In Section \ref{sec:HaarMultipliers} we prove Theorem \ref{t: HaarmultBMO}.
Finally, in Section \ref{sec:IndexedSpaces} we extend Theorem \ref{t: HaarmultBMO} to the whole scale of indexed spaces between little bmo and product BMO in the general multiparameter setting, with the appropriate iterated commutator in each case.

\textbf{Acknowledgements.} The authors would like to thank Professor Stefanie Petermichl for suggesting the problem of this paper, and for the valuable feedback and guidance in the process of writing it.

\section{Background and notation}
\label{sec:Background}

We collect here some notation, definitions and a few basic facts that will be used repeatedly in the sequel. While all of them are also valid in any multiparameter product space $\R^{d_1}\times\dots\times\R^{d_t}$, with the obvious modifications, for simplicity we restrict ourselves to the case of the product space $\R\times\R$.
 
\subsection{Dyadic intervals and dyadic rectangles} We denote by $\cD$ the set of all dyadic intervals in $\R$,
\begin{equation*}
\cD:=\lbrace [m2^{k},(m+1)2^{k}):~k,m\in\Z\rbrace.
\end{equation*}
We also denote by $\bfD$ the set of all dyadic rectangles in $\R^2$,
\begin{equation*}
\bfD:=\lbrace I\times J:~I,J\in\cD\rbrace.
\end{equation*}
For any $E\subseteq\R\times\R$ we denote
\begin{equation*}
\bfD(E):=\lbrace R\in\bfD:~ R\subseteq E\rbrace.
\end{equation*}
Note that if $E\in\bfD$, then $E\in\bfD(E)$.

\subsection{Haar systems}

\subsubsection{Haar system on \texorpdfstring{$\R$}{R}} For any $I\in\cD$, $h^{(0)}\ci{I}, h^{(1)}\ci{I}$ will denote, respectively, the $L^2$-normalized \emph{cancellative} and \emph{non-cancellative} Haar functions over the interval $I\in\cD$, that is
\begin{equation*}
h^{(0)}\ci{I}:=\frac{\1\ci{I_{+}}-\1\ci{I_{-}}}{\sqrt{|I|}},\qquad h^{(1)}\ci{I}:=\frac{\1\ci{I}}{\sqrt{|I|}}
\end{equation*}
(so $h^{(0)}\ci{I}$ has mean 0). For simplicity we denote $h\ci{I}:=h\ci{I}^{(0)}$. For any function $f\in L^1\ti{loc}(\R)$, we denote $f\ci{I}:=\La f,h\ci{I}\Ra$, $I\in\cD$. We will also denote by $Q\ci{I}$ the projection on the one-dimensional subspace spanned by $h\ci{I}$,
\begin{equation*}
Q\ci{I}f:=f\ci{I}h\ci{I},\qquad f\in L^1\ti{loc}(\R).
\end{equation*}
It is well-known that one has the expansion
\begin{equation*}
f=\sum_{I\in\cD}f\ci{I}h\ci{I},\qquad\forall f\in L^2(\R)
\end{equation*}
in the $L^2(\R)$-sense, and that the system $\lbrace h\ci{I}\rbrace\ci{I\in\cD}$ forms an orthonormal basis for $L^2(\R)$. It is then easy to see that for any $I\in\cD$ there holds
\begin{equation*}
\1\ci{I}(f-\La f\Ra\ci{I})=\sum_{\substack{J\in\cD\\J\subseteq I}}f\ci{J}h\ci{J}.
\end{equation*}

\subsubsection{Haar system on the product space \texorpdfstring{$\R\times\R$}{RxR}} If $R=I\times J$ is a dyadic rectangle in $\R^2$, we denote by $h\ci{R}^{(\e_1\e_2)}$ any of the four $L^2$-normalized Haar functions over $R$,
\begin{equation*}
h\ci{R}^{(\e_1\e_2)}:=h\ci{I}^{(\e_1)}\otimes h\ci{J}^{(\e_2)},\qquad \e_1,\e_2\in\lbrace0,1\rbrace,
\end{equation*}
that is
\begin{equation*}
h\ci{R}^{(\e_1\e_2)}(t,s)=h\ci{I}^{(\e_1)}(t)h\ci{J}^{(\e_2)}(s),\qquad (t,s)\in\R^2.
\end{equation*}
For simplicity we denote $h\ci{R}:=h\ci{R}^{(00)}$. For any function $f\in L^1\ti{loc}(\R^2)$, we denote
\begin{equation*}
f\ci{R}^{(\e_1\e_2)}:=\La f,h^{(\e_1\e_2)}\ci{R}\Ra,\qquad R\in\bfD,\qquad \e_1,\e_2\in\lbrace0,1\rbrace,
\end{equation*}
and we will often use the simplification $f\ci{R}:=\La f,h\ci{R}\Ra$. We will also denote by $Q\ci{R}$ the projection on the one-dimensional subspace spanned by $h\ci{R}$,
\begin{equation*}
Q\ci{R}f:=f\ci{R}h\ci{R},\qquad f\in L^1\ti{loc}(\R^2).
\end{equation*}
From the corresponding one-dimensional facts we immediately deduce the expansion
\begin{equation*}
f=\sum_{R\in\bfD}f\ci{R}h\ci{R},\qquad\forall f\in L^2(\R^2)
\end{equation*}
in the $L^2(\R^2)$-sense, and that the system $\lbrace h\ci{R}\rbrace\ci{R\in\bfD}$ forms an orthonormal basis for $L^2(\R^2)$. It is then easy to see by direct computation, following a reasoning similar to that of the inclusion-exclusion principle, that for any $R=I\times J\in\bfD$ there holds
\begin{equation*}
\sum_{R'\in\bfD(R)}f\ci{R'}h\ci{R'}(t,s)=\1\ci{R}(t,s)(f(t,s)-\La f(\cdot,s)\Ra\ci{I}-\La f(t,\fdot)\Ra\ci{J}+\La f\Ra\ci{R}).
\end{equation*}

Finally, for $I,J\in\cD$ we denote by $Q\ci{I}^1,Q\ci{J}^2$ the operators acting on functions $f\in L^1\ti{loc}(\R^2)$ by
\begin{equation*}
Q\ci{I}^{1}f(t,s)=Q\ci{I}(f(\fdot,s))(t),\qquad Q\ci{J}^{2}f(t,s)=Q\ci{J}(f(t,\fdot))(s),\qquad\text{for a.e. }(t,s)\in\R\times\R.
\end{equation*}
Thus, if $R=I\times J$ then $Q\ci{R}=Q^{1}\ci{I}Q^2\ci{J}=Q^{2}\ci{J}Q^{1}\ci{I}$. Note that $(Q\ci{I}^1)^2=Q\ci{I}^1$ and $(Q\ci{J}^2)^2=Q\ci{J}^2$.

\subsection{\texorpdfstring{$A_p$}{Ap} weights}

By weight we always mean a locally integrable, a.e. positive function. We fix in what follows $1<p<\infty$.
Given a weight $w,$ we consider the weighted Lebesgue space $L^p(w),$ that is the space of $p$-integrable functions with respect to the measure $w(x)\, dx.$
In other words, we say that a function $f$ belongs to $L^p(w)$ if
\begin{equation*}
    \Vert f \Vert\ci{L^p(w)} := \left(\int |f(x)|^p w(x)\, dx\right)^{1/p} < \infty.
\end{equation*}
The dual of $L^p(w)$ is the weighted space $L^{p'}(w')$ under the usual $L^2$-pairing $\La f,g \Ra = \int f(x) \overline{g(x)}\, dx,$ where $p'$ is the H\ddoto lder conjugate of $p$ (that is $\frac{1}{p} + \frac{1}{p'} = 1$) and $w'(x) = (w(x))^{1-p'}$ is the conjugate weight to $w.$
Note here that, in the particular case $p = 2,$ one has that the dual of $L^2(w)$ is $L^2(w^{-1}).$

\subsubsection{\texorpdfstring{$A_p$}{Ap} weights on \texorpdfstring{$\R$}{R}} Consider a weight $w$ on $\R$. It should be noted that all that follows would still hold if we considered weights on $\R^d,$ for any $d$, just by  substituting intervals of $\R$ by cubes of $\R^d.$ We define the \emph{Muckenhoupt $A_p$ characteristic} of $w,$ denoted by $[w]\ci{A_p},$ as
\begin{equation*}
[w]\ci{A_p}:=\sup_{I}~\La w\Ra\ci{I}\La w^{-1/(p-1)}\Ra\ci{I}^{p-1} = \sup_{I}~\La w\Ra\ci{I}\La w'\Ra\ci{I}^{p-1},
\end{equation*}
where the supremum is taken over all intervals $I$ in $\R$. We define a dyadic version of this by
\begin{equation*}
[w]\ci{A_p,\cD}:=\sup_{I\in\cD}~\La w\Ra\ci{I}\La w^{-1/(p-1)}\Ra\ci{I}^{p-1}.
\end{equation*}
We say that $w$ is an $A_p$ weight, respectively a dyadic $A_p$ weight, if $[w]\ci{A_p}<\infty$, respectively $[w]\ci{A_p,\cD}<\infty$. Note that a very easy application of Jensen's inequality gives $[w]\ci{A_p,\cD}\geq1$, see \cite[Lemma 4.1]{Convex} for details.
Observe as well that $w$ is an $A_p$ weight if and only if $w'$ is an $A_{p'}$ weight, and in this case $[w']\ci{A_{p'}} = [w]\ci{A_p}^{p'-1}.$
The analogous fact is also true in the dyadic case.

It is a classical result that $[w]\ci{A_p}<\infty$ if and only if the Hardy--Littlewood maximal function $M$ given by
\begin{equation*}
Mf:=\sup_{I}~\La|f|\Ra\ci{I}\1\ci{I},
\end{equation*}
where supremum is taken over all intervals $I$ in $\R$, is bounded as an operator from $L^{p}(w)$ into itself, and that in fact one has the estimate
\begin{equation}
\label{Ap-maximal-function-onepar}
\Vert M\Vert\ci{L^{p}(w)\rightarrow L^{p}(w)}\lesssim_{p}[w]\ci{A_p}^{1/(p-1)}.
\end{equation}
A dyadic version of this is also true for the dyadic Hardy--Littlewood maximal function $M\ci{\cD}$ given by
\begin{equation*}
M\ci{\cD}f:=\sup_{I\in\cD}~\La|f|\Ra\ci{I}\1\ci{I}.
\end{equation*}

\subsubsection{Biparameter \texorpdfstring{$A_p$}{Ap} weights on \texorpdfstring{$\R^2$}{R2}} Consider now a weight $w$ on $\R\times\R$. As before, all that follows would be equally valid for weights on any $\R^{\vec{d}} := \R^{d_1}\times\dots\times\R^{d_t}$, with the obvious modifications. We define the \emph{biparameter Muckenhoupt $A_p$ characteristic} $[w]\ci{A_p}$ of $w$ by
\begin{equation*}
[w]\ci{A_p}:=\sup_{R}~\La w\Ra\ci{R}\La w^{-1/(p-1)}\Ra\ci{R}^{p-1} = \sup_{R}~\La w\Ra\ci{R}\La w'\Ra\ci{R}^{p-1},
\end{equation*}
where supremum is taken over \emph{all} rectangles in $\R\times\R$ (with sides parallel to the coordinate axis). We define a dyadic version of this by
\begin{equation*}
[w]\ci{A_p,\bfD}:=\sup_{R\in\bfD}~\La w\Ra\ci{R}\La w^{-1/(p-1)}\Ra\ci{R}^{p-1}.
\end{equation*}
We say that $w$ is a biparameter $A_p$ weight, respectively a dyadic biparameter $A_p$ weight, if $[w]\ci{A_p}<\infty$, respectively $[w]\ci{A_p,\bfD}<\infty$. Note that similarly to the one-parameter case we have $[w]\ci{A_p}\geq1$ and $[w']\ci{A_{p'}} = [w]\ci{A_p}^{p'-1},$ as well as the analogous facts for dyadic $A_p$ and dyadic $A_{p'}$ weights.

Consider the \emph{strong} maximal function $M\ti{S}$ given by
\begin{equation*}
M\ti{S}f:=\sup_{R}~\La|f|\Ra\ci{R}\1\ci{R},
\end{equation*}
where as previously the supremum is taken over \emph{all} rectangles in $\R\times\R$ (with sides parallel to the coordinate axis). Consider also the Hardy--Littlewood maximal functions acting in each variable separately,
\begin{equation*}
M^1f(t,s):=M(f(\fdot,s))(t),\qquad M^2f(t,s):=M(f(t,\fdot))(s),\qquad f\in L^1\ti{loc}(\R^2).
\end{equation*}
Using the Lebesgue Differentiation Theorem it is easy to see that
\begin{equation*}
[w]\ci{A_p}\geq\max(\esssup_{x_1\in\R}[w(x_1,\fdot)]\ci{A_p},\esssup_{x_2\in\R}[w(\fdot,x_2)]\ci{A_p}).
\end{equation*}
It is also easy to see that $M\ti{S}f\leq M^1(M^2f)$, which coupled with \eqref{Ap-maximal-function-onepar} implies immediately
\begin{equation*}
\Vert M\ti{S}\Vert\ci{L^{p}(w)\rightarrow L^{p}(w)}\lesssim_{p}\esssup_{x_1\in\R}[w(x_1,\fdot)]\ci{A_p}^{1/(p-1)}\cdot \esssup_{x_2\in\R}[w(\fdot,x_2)]\ci{A_p}^{1/(p-1)}.
\end{equation*}
On the other hand, it is also immediate to see that
\begin{equation*}
\Vert M\ti{S}\Vert\ci{L^{p}(w)\rightarrow L^{p}(w)}\geq[w]\ci{A_{p}}^{1/p}.
\end{equation*}
Therefore, $w$ is a biparameter $A_p$ weight if and only if $w$ is an $A_p$ weight in each variable separately and uniformly. Dyadic versions of these facts are similarly true for the \emph{dyadic strong} maximal function $M\ci{\bfD}$ given by
\begin{equation*}
M\ci{\bfD}f:=\sup_{R\in\bfD}~\La|f|\Ra\ci{R}\1\ci{R}.
\end{equation*}

Here it is worth mentioning the analogous property for $A_p$ weights defined on $\R^{\vec{d}} = \R^{d_1} \times \dots \times \R^{d_t}.$
We first introduce some convenient notation for the general setting.
Given $x \in \R^{\vec{d}}$ and $k \in \lbrace 1, 2, \ldots, t\rbrace,$ let us denote $x_{\overline{k}} := (x_1,\ldots,x_{k-1},\cdot,x_{k+1},\ldots,x_t),$ so that a function $f(x_{\overline{k}})$ depends only on the variable $x_k.$
Similarly, we denote $\R^{\vec{d}_{\overline{k}}} := \R^{d_1} \times \dots \times \R^{d_{k-1}} \times \R^{d_{k+1}} \times \dots \times \R^{d_{t}}.$
In general, we will also extend this notation to any number of parameters, so that if $k_1,\ldots,k_s \in \lbrace 1, 2, \ldots, t\rbrace,$ then we will denote by $x_{\overline{k_1,\ldots,k_s}} := (x_1,\ldots,x_{k_1-1},\cdot,x_{k_1+1},\ldots,x_{k_s-1},\cdot,x_{k_s+1},\ldots,x_t),$ and similarly for $\R^{\vec{d}_{\overline{k_1,\ldots,k_s}}}.$
Then, the same reasoning as before shows that a weight $w$ is a multiparameter $A_p$ weight on $\R^{\vec{d}}$ if and only if, for any subsequence $k = (k_1, \ldots, k_s)$ of $(1, \ldots, t),$ the weight $w(x\ci{\overline{k}})$ is an $A_p$ weight on $\R^{\vec{d}\ci{\overline{k}}}$ with $A_p$ characteristic uniformly bounded on $x \in \R^{\vec{d}}.$
In particular, it is easy to see that
\begin{equation*}
  [w]\ci{A_p} \geq
  \max_k \bigg( \esssup_{x_{\overline{k}} \in \R^{\vec{d}_{\overline{k}}}} [w(x_{\overline{k}})]\ci{A_p} \bigg),
\end{equation*}
where the maximum ranges over all subsequences $k$ of $(1,\ldots,t).$

\subsubsection{Averages of \texorpdfstring{$A_p$}{Ap} weights} We recall a few standard facts about averages of $A_p$ weights. Let $\mu,\lambda$ be biparameter $A_p$ weights on $\R\times\R$. Using several times Jensen's inequality, H\ddoto lder's inequality and the $A_p$ condition for the weights $\mu$ and $\lambda$ it is easy to see that for all rectangles $R$ one has the estimates
\begin{equation}
\label{averages Ap}
\La \mu^{1/p}\Ra\ci{R}\sim\ci{[\mu]\ci{A_p}}\La\mu\Ra\ci{R}^{1/p}\sim\ci{[\mu]\ci{A_p}}\La \mu^{-1/(p-1)}\Ra\ci{R}^{-(p-1)/p}\sim\ci{[\mu]\ci{A_p}}\La \mu^{-1/p}\Ra\ci{R}^{-1},
\end{equation}
and
\begin{equation}
\label{multiplied averages Ap}
\La \mu^{1/p}\lambda^{-1/p}\Ra\ci{R}\sim\ci{[\mu]\ci{A_p},[\lambda]\ci{A_p}}\La\mu\Ra\ci{R}^{1/p}\La\lambda\Ra\ci{R}
^{-1/p},
\end{equation}
see \cite[~p. 2]{treil et al} for a sketch of the argument (only the one-parameter setting is treated there, but it is obvious that the same argument works in the multiparameter setting without any changes at all). Moreover, using H\ddoto lder's inequality it is easy to see that
\begin{equation*}
1\leq [\mu^{1/p}\lambda^{-1/p}]\ci{A_2}\leq[\mu]\ci{A_p}^{1/p}[\lambda]\ci{A_p}^{1/p},
\end{equation*}
see \cite[Lemma 2.7]{holmes-lacey-wick} for a full proof (again, while only the one-parameter setting is treated there, the multiparameter result follows from the same arguments).

Note that dyadic versions of all the above facts are similarly true.

\subsection{Dyadic square functions and Littlewood--Paley estimates} We denote by $S\ci{\cD}$ the dyadic square function in $\R$,
\begin{equation*}
S\ci{\cD}f:=\left(\sum_{I\in\cD}|Q\ci{I}f|^2\right)^{1/2}=\left(\sum_{I\in\cD}|f\ci{I}|^2\frac{\1\ci{I}}{|I|}\right)^{1/2},\qquad f\in L^1\ti{loc}(\R).
\end{equation*}
We also denote by $S\ci{\bfD}$ the dyadic biparameter square function in $\R\times\R$,
\begin{equation*}
S\ci{\bfD}f:=\left(\sum_{R\in\bfD}|Q\ci{R}f|^2\right)^{1/2}=\left(\sum_{R\in\bfD}|f\ci{R}|^2\frac{\1\ci{R}}{|R|}\right)^{1/2},\qquad f\in L^1\ti{loc}(\R^2).
\end{equation*}
It is well-known that if $w$ is a dyadic $A_p$ weight on $\R$, $1<p<\infty$, then
\begin{equation*}
\Vert S\ci{\cD}f\Vert\ci{L^{p}(w)}\sim\ci{p,[w]\ci{A_p,\cD}}\Vert f\Vert\ci{L^{p}(w)},
\end{equation*}
for all (suitable) functions $f$ on $\R$. Iterating this and using well-known results about vector-valued extensions of linear operators (see e.g. \cite[Chapter 5]{grafakos-classical}) we deduce, as remarked in \cite{holmes-petermichl-wick}, that if $w$ is a dyadic biparameter $A_p$ weight on $\R\times\R$, then 
\begin{equation*}
\Vert S\ci{\bfD}f\Vert\ci{L^{p}(w)}\sim\ci{p,[w]\ci{A_p,\bfD}}\Vert f\Vert\ci{L^{p}(w)},
\end{equation*}
for all (suitable) functions $f$ on $\R^2$ (e.g. $f\in L^{\infty}\ti{c}(\R^2)$ suffices, and then using approximation arguments one can extend it to more general functions $f$). In particular, the set of all finite linear combinations of (bi-cancellative) Haar functions in $\R^2$ is dense in $L^{p}(w)$.

Note that for $p=2$ we simply have
\begin{equation*}
\Vert S\ci{\bfD}f\Vert\ci{L^2(w)}=\left(\sum_{R\in\bfD}|f\ci{R}|^2\La w\Ra\ci{R}\right)^{1/2}.
\end{equation*}

The above results also hold in the multiparameter setting, with the usual modifications for general product spaces.

For the general multiparameter case, we will also need to consider \emph{indexed square functions} for functions defined on $\R^{\vec{d}} = \R^{d_1} \times \dots \times \R^{d_t}.$
For $i \in \lbrace 1, \ldots, t \rbrace$ and $I_i \in \cD(\R^{d_i}),$ we denote here by $Q\ci{I_i}^i$ the operator $Q\ci{I_i}$ acting on the $i$-th variable, that is $Q\ci{I_i}^i f(x) = Q\ci{I_i}(f(x_{\overline{i}}))(x_i)$ for $f \in L\ti{loc}^1(\R^{\vec{d}}).$
Similarly, for a subsequence $k = (k_1, \ldots, k_s)$ of $(1, \ldots, t),$ let $R = I\ci{k_1} \times \dots \times I\ci{k_s} \in \bfD(\R^{\vec{d}\ci{\overline{k}}})$ and denote $Q\ci{R}^k = Q\ci{I_1}^{k_1} \dots Q\ci{I_s}^{k_s}.$
Fix a subsequence $k = (k_1, \ldots, k_s)$ of $(1, \ldots, t).$
Then, we define the $k$\emph{-indexed square function} $S\ci{\bfD}^k f$ of $f$ by
\begin{equation*}
  S\ci{\bfD}^k f :=
  \left(\sum_{R\in\bfD} |Q\ci{R}^k f|^2 \right)^{1/2} =
  \left(\sum_{R\in\bfD} |f\ci{R}^k|^2 \frac{\1\ci{R}^k}{|R|}\right)^{1/2}, \qquad f\in L^1\ti{loc}(\R^{\vec{d}})
\end{equation*}
(here we make an abuse of the notation $\bfD,$ relying on the context for understanding to which space it refers).
Then, using an $A\ci{\infty}$-extrapolation result due to D.~Cruz-Uribe, J.~M.~Martell and C.~Pérez \cite[Theorem~2.1]{cruz-uribe-martell-perez} one can see that for any subsequence $k$ of $(1, \ldots, t),$ any $1 < p < \infty$ and any dyadic $t$-parameter $A_p$ weight $w$ it holds that
\begin{equation*}
  \Vert S\ci{\bfD}^k f \Vert\ci{L^p(w)} \sim\ci{\vec{d},p,[w]\ci{A_p,\bfD}}
  \Vert f \Vert\ci{L^p(w)}
\end{equation*}
(see also \cite[Lemma~2.2]{airta}).

\subsubsection{Dyadic square function over collections of dyadic rectangles}

Let $\cU$ be any collection of dyadic rectangles in $\R\times\R$. We denote
\begin{equation*}
P\ci{\cU}f:=\sum_{R\in\cU}f\ci{R}h\ci{R},\qquad S\ci{\cU}f:=\left(\sum_{R\in\cU}|f\ci{R}|^2\frac{\1\ci{R}}{|R|}\right)^{1/2}.
\end{equation*}
By the above we have that if $w$ is a biparameter $A_p$ weight on $\R\times\R$ then there holds
\begin{align*}
\Vert P\ci{\cU}f\Vert\ci{L^p(w)}\sim\ci{p,[w]\ci{A_p,\bfD}}\Vert S\ci{\cU}f\Vert\ci{L^{p}(w)}
\leq\Vert S\ci{\bfD}f\Vert\ci{L^{p}(w)}\sim\ci{p,[w]\ci{A_p,\bfD}}\Vert f\Vert\ci{L^p(w)}.
\end{align*}
In particular
\begin{equation*}
\Vert Q\ci{R}f\Vert\ci{L^p(w)}\lesssim\ci{[w]\ci{A_p,\bfD}}\Vert f\Vert\ci{L^p(w)},\qquad\forall R\in\bfD.
\end{equation*}
If $\Omega$ is any subset of $\R^2$, we denote
\begin{equation*}
P\ci{\Omega}f:=P\ci{\bfD(\Omega)}f.
\end{equation*}
Again, all these statements hold in the multiparameter setting as well.

\subsubsection{Incorporating the weight in the square function}

It is sometimes convenient to define square functions in an alternative way that directly incorporates the weight in the operator. Namely, given $1<p<\infty$ and any dyadic biparameter $A_p$ weight $w$ on $\R^2$, define
\begin{equation*}
S\ci{w}f:=\left(\sum_{R\in\bfD}|f\ci{R}|^2\La w\Ra\ci{R}^{2/p}\frac{\1\ci{R}}{|R|}\right)^{1/2},\qquad f\in L^1\ti{loc}(\R^2).
\end{equation*}
This type of square functions appears naturally in the theory of matrix-valued weights (see e.g. \cite{isralowitz} for estimates in the matrix-weighted one-parameter setting), and is sometimes referred to as the Triebel--Lizorkin square function associated to $w$, because $L^{p}$ bounds for it allow one to identify $L^{p}(w)$ as a certain Trielel--Lizorkin space (see e.g. \cite{frazier-jawerth-weiss} for the scalar one-parameter case). Here we prove an estimate in the scalar biparameter setting. In fact, the proofs of Lemma \ref{l: tl-upper bound} and Corollary \ref{c: tl-lower bound} below readily extend to the matrix-valued setting. This will be part of forthcoming work of the authors.

\begin{lm}
\label{l: tl-upper bound}
Let $f\in L^1\ti{loc}(\R^2)$. Then
\begin{equation*}
\Vert S_{w}f\Vert\ci{L^{p}}\lesssim_{p,[w]\ci{A_p,\bfD}}\Vert Sf\Vert\ci{L^{p}(w)}.
\end{equation*}
In particular
\begin{equation*}
\Vert S_{w}f\Vert\ci{L^p}\lesssim_{p,[w]\ci{A_p,\bfD}}\Vert f\Vert\ci{L^{p}(w)},\qquad\forall f\in L^{\infty}\ti{c}(\R^2).
\end{equation*}
\end{lm}

\begin{proof}
First of all, we note that by the Monotone Convergence Theorem we can assume without loss of generality that $f$ has only finitely-many non-zero Haar coefficients. Now, we notice that by \eqref{averages Ap} we have
\begin{align*}
\Vert S_{w}f\Vert\ci{L^p}^{p}&=\int_{\R^2}\bigg(\sum_{R\in\bfD}|f\ci{R}|^2\La w\Ra\ci{R}^{2/p}\frac{\1\ci{R}(x)}{|R|}\bigg)^{p/2}\mathd x\\
&\lesssim_{p,[w]\ci{A_p,\bfD}}\int_{\R^2}\bigg(\sum_{R\in\bfD}| f\ci{R}|^2(\La w^{1/p}\Ra\ci{R})^2\frac{\1\ci{R}(x)}{|R|}\bigg)^{p/2}\mathd x.
\end{align*}
Thus, by standard (unweighted) dyadic Littlewood--Paley theory we only have to prove that
\begin{align*}
\Vert F\Vert\ci{L^{p}}:=\bigg\Vert \sum_{R\in\bfD}|f\ci{R}|\La w^{1/p}\Ra\ci{R}h\ci{R}\bigg\Vert\ci{L^p}\lesssim_{p} \Vert Sf\Vert\ci{L^{p}(w)}.
\end{align*}

We use duality for that. Let $g\in L^{p'}$. Then, we have
\begin{align*}
|\La F,g\Ra|&\leq\sum_{R\in\bfD} |f\ci{R}|\La w^{1/p}\Ra\ci{R}\cdot|g\ci{R}|=\int_{\R^2}\sum_{R\in\bfD}|f\ci{R}|w(x)^{1/p}\cdot h\ci{R}(x)\cdot|g\ci{R}|\cdot h\ci{R}(x)\mathd x\\
&\leq\int_{\R^2}(Sf)(x)w(x)^{1/p}(Sg)(x)\mathd x\leq \Vert Sf\Vert\ci{L^{p}(w)}\Vert Sg\Vert\ci{L^{p'}}
\sim_{p}\Vert Sf\Vert\ci{L^{p}(w)}\Vert g\Vert\ci{L^{p'}},
\end{align*}
concluding the proof.
\end{proof}

Of course, for $p=2$ we have just $\Vert S_{w}f\Vert\ci{L^2}=\Vert Sf\Vert\ci{L^2(w)}$.

Using Lemma \ref{l: tl-upper bound}, we can immediately deduce the corresponding lower bound.

\begin{cor}
\label{c: tl-lower bound}
There holds
\begin{equation*}
\Vert f\Vert\ci{L^p(w)}\lesssim_{p,[w]\ci{A_p,\bfD}}\Vert S_{w}f\Vert\ci{L^p},\qquad\forall f\in L^{\infty}\ti{c}(\R^{2}).
\end{equation*}
\end{cor}

\begin{proof}
We use duality. Recall that $w':=w^{-1/(p-1)}$ is a dyadic biparameter $A_{p'}$ weight with $[w]\ci{A_{p'}}^{1/p'}=[w]\ci{A_p}^{1/p}$. Let $f,g\in L^{\infty}\ti{c}(\R^{2})$. Then, using \eqref{averages Ap} and applying Lemma \ref{l: tl-upper bound} (for the weight $w'$) we get
\begin{align*}
|\La f,g\Ra|&\leq\sum_{R\in\bfD}| f\ci{R}|\cdot|g\ci{R}|=
\sum_{R\in\bfD}\int_{\R^2}|f\ci{R}|\La w\Ra\ci{R}^{1/p}h\ci{R}(x)\cdot|g\ci{R}|\La w\Ra\ci{R}^{-1/p}h\ci{R}(x)\mathd x\\
&\lesssim_{p,[w]\ci{A_p,\bfD}}
\sum_{R\in\bfD}\int_{\R^2}|f\ci{R}|\La w\Ra\ci{R}^{1/p}h\ci{R}(x)\cdot|g\ci{R}|\La w'\Ra\ci{R}^{1/p'}h\ci{R}(x)\mathd x\\
&\leq\int_{\R^{2}}(S_{w}f(x))(S_{w'}g(x))\mathd x\leq\Vert S_wf\Vert\ci{L^{p}}\Vert S_{w'}g\Vert\ci{L^{p'}}\\
&\lesssim_{p,[w]\ci{A_p,\bfD}}\Vert S_{w}f\Vert\ci{L^{p}}\Vert g\Vert\ci{L^{p'}(w')},
\end{align*}
so $\Vert f\Vert\ci{L^{p}(w)}\lesssim_{p,[w]\ci{A_{p},\bfD}}\Vert S_{w}f\Vert\ci{L^{p}}$, concluding the proof.
\end{proof}

It is clear that the above results remain true in the general multiparameter setting.

\section{Equivalences for dyadic Bloom product BMO}
\label{sec:WeightedEquivalences}

We prove in this section that the dyadic Bloom product BMO admits several equivalent descriptions.
We will focus on dyadic biparameter $A_p$ weights, with $1 < p <\infty,$ on $\R\times\R.$
It should be noted that the results presented here also hold in the general multiparameter setting of functions and weights defined on $\R^{d_1}\times\dots\times\R^{d_t}$, with identical proofs.
However, for simplicity we will restrict ourselves to the case of $\R\times\R.$

In the sequel we fix $1<p<\infty$ and dyadic biparameter $A_p$ weights $\mu,\lambda$ on $\R^2$, and we set $\nu:=\mu^{1/p}\lambda^{-1/p}$. Note that we will be systematically suppressing from the notation dependence of implied constants on the value of $p$ and the Muckenhoupt characteristics $[\mu]\ci{A_p,\bfD}$ and $[\lambda]\ci{A_p,\bfD}$. Recall that $\nu$ is a dyadic biparameter $A_2$ weight with $1\leq[\nu]\ci{A_2}\leq[\mu]\ci{A_p}^{1/p}[\lambda]\ci{A_p}^{1/p}$.

Given any sequence $a=\lbrace a\ci{R}\rbrace\ci{R\in\bfD}$ of complex numbers, we define the \emph{dyadic two-weight Bloom product BMO $p$-norm} $\Vert a\Vert\ci{\text{BMO}\ci{\text{prod},\bfD}(\mu,\lambda,p)}$ by
\begin{equation*}
\Vert a\Vert\ci{\text{BMO}\ci{\text{prod},\bfD}(\mu,\lambda,p)}:=\sup_{\cU}\frac{1}{(\mu(\text{sh}(\cU)))^{1/p}}\left\Vert \left(\sum_{R\in\cU}|a\ci{R}|^2\frac{\1\ci{R}}{|R|}\right)^{1/2}\right\Vert\ci{L^p(\lambda)},
\end{equation*}
where the supremum ranges over all non-empty collections $\cU$ of dyadic rectangles in $\R\times\R$, and
\begin{equation*}
\text{sh}(\cU):=\bigcup_{R\in\cU}R.
\end{equation*}
By the Monotone Convergence Theorem, it is clear that one can restrict the supremum in the definition of $\Vert a\Vert\ci{\text{BMO}\ci{\text{prod},\bfD}(\mu,\lambda,p)}$ to just non-empty finite sucollections $\cU$ of $\bfD$. Moreover, it is easy to see that
\begin{equation*}
\Vert a\Vert\ci{\text{BMO}\ci{\text{prod},\bfD}(\mu,\lambda,p)}=
\sup_{\Omega}\frac{1}{(\mu(\Omega))^{1/p}}\left\Vert\left( \sum_{R\in\bfD(\Omega)}|a\ci{R}|^2\frac{\1\ci{R}}{|R|}\right)^{1/2}\right\Vert\ci{L^p(\lambda)},
\end{equation*}
where the supremum is taken over all non-empty bounded open sets $\Omega$ in $\R^2$, or even over all measurable subsets $\Omega$ of $\R^2$ of non-zero finite measure.

For any dyadic biparameter $A_2$ weight $\nu$ on $\R^2$, define the \emph{dyadic one-weight Bloom product BMO norm} $\Vert a\Vert\ci{\text{BMO}\ci{\text{prod},\bfD}(\nu)}$ by
\begin{equation*}
\Vert a\Vert\ci{\text{BMO}\ci{\text{prod},\bfD}(\nu)}:=\Vert a\Vert\ci{\text{BMO}\ci{\text{prod},\bfD}(\nu,\nu^{-1},2)}.
\end{equation*}
We also define the \emph{unweighted product BMO norm}
\begin{equation*}
\Vert a\Vert\ci{\text{BMO}\ci{\text{prod},\bfD}}:=\Vert a\Vert\ci{\text{BMO}\ci{\text{prod},\bfD}(1)}.
\end{equation*}

If $b\in L^1\ti{loc}(\R)$, then we define
\begin{equation*}
\Vert b\Vert\ci{\text{BMO}\ci{\text{prod},\bfD}(\mu,\lambda,p)}:=\Vert\lbrace b\ci{R}\rbrace\ci{R\in\bfD}\Vert\ci{\text{BMO}\ci{\text{prod},\bfD}(\mu,\lambda,p)}.
\end{equation*}

The main goal of this section is to prove Theorem \ref{t: equivBMO}, which we state again for the reader's convenience, in a slightly more general (but in view of the Monotone Convergence Theorem, equivalent) form.

\theoremstyle{plain}
\newtheorem*{thm:equivBMO*}{Theorem \ref{t: equivBMO}}
\begin{thm:equivBMO*}
Let $1 < p <\infty$.
Consider dyadic biparameter $A_p$ weights $\mu,$ $\lambda$ and define $\nu:=\mu^{1/p}\lambda^{-1/p}$. Then
\begin{equation*}
\Vert a\Vert\ci{\emph{BMO}\ci{\emph{prod},\bfD}(\nu)}\sim\Vert a\Vert\ci{\emph{BMO}\ci{\emph{prod},\bfD}(\mu,\lambda,p)},
\end{equation*}
for any sequence of complex numbers $a=\lbrace a\ci{R}\rbrace\ci{R\in\bfD}$, where the implied constants depend only on $p,$ $[\mu]\ci{A_p,\bfD}$ and $[\lambda]\ci{A_p,\bfD}$.
\end{thm:equivBMO*}

The one-parameter analog of Theorem \ref{t: equivBMO} was shown by I.~Holmes, M.~T.~Lacey and B.~D.~Wick \cite{holmes-lacey-wick}. Moreover, the little bmo analog of Theorem \ref{t: equivBMO} was established by I.~Holmes, S.~Petermichl and B.~D.~Wick \cite{holmes-petermichl-wick} by iterating the one-parameter result of \cite{holmes-lacey-wick}. Also, in the special case $\mu=\lambda$, Theorem \ref{t: equivBMO} was proved by E.~Airta, K.~Li, H.~Martikainen and E.~Vuorinen \cite{martikainen et al} (in fact under the weaker assumption that $\mu=\lambda$ is just a biparameter $A_{\infty}$ weight).

The proof of Theorem \ref{t: equivBMO} will be done in several steps. Note that by the Monotone Convergence Theorem, we can without loss of generality assume that the sequence $a=\lbrace a\ci{R}\rbrace\ci{R\in\bfD}$ is finitely supported.

Now, consider the ``purely non-cancellative'' biparameter paraproduct with symbol $a$,
\begin{equation*}
\Pi_{a}^{(11)}f:=\sum_{R\in\bfD}a\ci{R}\La f\Ra\ci{R}h\ci{R}.
\end{equation*}
Note that the superscript $(11)$ indicates that in the expression of the paraproduct $f$ is integrated only against Haar functions of the form $h^{(11)}\ci{R}$.
This operator is defined and used by Holmes--Lacey--Wick \cite{holmes-lacey-wick} in the Bloom one-parameter setting, as well as by Blasco--Pott \cite{blasco-pott} in the unweighted biparameter setting (for $p=2$) and Holmes--Petermichl--Wick \cite{holmes-petermichl-wick}  in the Bloom biparameter setting.

First note that just by testing $\Pi_{a}^{(1,1)}$ on $\1\ci{\sh(\cU)}$ and then using the weighted Littlewood--Paley estimates, for any $\cU\subseteq\bfD$, we trivially deduce
\begin{equation*}
\Vert a\Vert\ci{\text{BMO}\ci{\text{prod},\bfD}(\mu,\lambda,p)}\lesssim\Vert \Pi^{(1,1)}_{a}\Vert\ci{L^{p}(\mu)\rightarrow L^{p}(\lambda)}.
\end{equation*}
We first show that the previous two quantities are actually equivalent in the regime $1 < p \leq 2$. To this end, we follow a scheme similar to that of one of the standard proofs of the dyadic Carleson's embedding theorem (see for instance \cite{treildyadic} for a very general one-parameter version), but with the one-parameter maximal function replaced by the dyadic strong maximal function.

\begin{prop}
\label{prop:ParaproductTwoWeightEquiv}
Assume $1 < p \leq 2.$
Then, there holds
\begin{equation*}
\Vert a\Vert\ci{\emph{BMO}\ci{\emph{prod},\bfD}(\mu,\lambda,p)}\sim\Vert \Pi^{(1,1)}_{a}\Vert\ci{L^{p}(\mu)\rightarrow L^{p}(\lambda)},
\end{equation*}
where the implied constants depend only on $p,[\mu]\ci{A_p,\bfD}$ and $[\lambda]\ci{A_p,\bfD}$.
\end{prop}

\begin{proof}
We only need to show the $\gtrsim$ direction, that is the inequality
\begin{equation*}
\Vert \Pi^{(1,1)}_{a}(f)\Vert\ci{L^{p}(\lambda)}\lesssim\Vert a\Vert\ci{\text{BMO}\ci{\text{prod},\bfD}(\mu,\lambda,p)}\Vert f\Vert\ci{L^{p}(\mu)}.
\end{equation*}
By the weighted Littlewood--Paley estimates, we have
\begin{align*}
\Vert \Pi_{a}^{(1,1)}(f)\Vert\ci{L^{p}(\lambda)}^{p}\sim\int_{\R\times\R}\left(\sum_{R\in\bfD}|a\ci{R}|^2|\La f\Ra\ci{R}|^2\frac{\1\ci{R}(x)}{|R|}\right)^{p/2}\lambda(x)dx.
\end{align*}
Thus, without loss of generality we may assume that $f\geq0$.

Fix $x\in\R\times\R$. Consider the measure $m$ on the countable set $\bfD$ given by
\begin{equation*}
m(R):=|a\ci{R}|^2\frac{\1\ci{R}(x)}{|R|},\qquad\forall R\in\bfD.
\end{equation*}
Consider also the function $g:\bfD\rightarrow[0,\infty)$ given by $g(R):=(\La f\Ra\ci{R})^{p}$, for all $R\in\bfD$. Then, we have
\begin{equation*}
\left(\sum_{R\in\bfD}|a\ci{R}|^2(\La f\Ra\ci{R})^2\frac{\1\ci{R}(x)}{|R|}\right)^{p/2}=
\Vert g\Vert\ci{L^{2/p}(\bfD,m)}.
\end{equation*}
We emphasize here that $2/p\geq 1$. Thus, in the scale of Lorentz spaces, we have (see e.g. \cite[Proposition 1.4.10]{grafakos-classical})
\begin{equation*}
L^{2/p,1}\subseteq L^{2/p,2/p}=L^{2/p}.
\end{equation*}
Therefore
\begin{align*}
\left(\sum_{R\in\bfD}|a\ci{R}|^2(\La f\Ra\ci{R})^2\frac{\1\ci{R}(x)}{|R|}\right)^{p/2}&=
\Vert g\Vert\ci{L^{2/p}(\bfD,m)}\lesssim_{p}\Vert g\Vert\ci{L^{2/p,1}(\bfD,m)}\\&
\sim_{p}\int_{0}^{\infty}(m(\lbrace g>t\rbrace))^{p/2}dt
=\int^{\infty}_{0}\bigg(\sum_{\substack{R\in\bfD\\(\La f\Ra\ci{R})^{p}>t}}|a\ci{R}|^2\frac{\1\ci{R}(x)}{|R|}\bigg)^{p/2}dt.
\end{align*}
Therefore
\begin{align*}
&\Vert \Pi_{a}^{(1,1)}(f)\Vert\ci{L^{p}(\lambda)}^{p}\lesssim_{p}\int_{\R\times\R}\lambda(x)\int_{0}^{\infty}\bigg(\sum_{\substack{R\in\bfD\\(\La f\Ra\ci{R})^{p}>t}}|a\ci{R}|^2\frac{\1\ci{R}(x)}{|R|}\bigg)^{p/2}dtdx\\
&\leq\int^{\infty}_{0}\int_{\R\times\R}\bigg(\sum_{\substack{R\in\bfD\\R\subseteq\lbrace(M\ci{\bfD}f)^{p}>t\rbrace}}|a\ci{R}|^2\frac{\1\ci{R}(x)}{|R|}\bigg)^{p/2}\lambda(x)dxdt\\
&\leq\Vert a\Vert\ci{\text{BMO}\ci{\text{prod},\bfD}(\mu,\lambda,p)}^{p}\int^{\infty}_{0}\mu(\lbrace (M\ci{\bfD}f)^p>t\rbrace)dt\\
&=\Vert a\Vert\ci{\text{BMO}\ci{\text{prod},\bfD}(\mu,\lambda,p)}^{p}\Vert (M\ci{\bfD}f)^{p}\Vert\ci{L^1(\mu)}=\Vert a\Vert\ci{\text{BMO}\ci{\text{prod},\bfD}(\mu,\lambda,p)}^{p}\Vert M\ci{\bfD}f\Vert\ci{L^p(\mu)}^p\\
&\lesssim \Vert a\Vert\ci{\text{BMO}\ci{\text{prod},\bfD}(\mu,\lambda,p)}^{p}\Vert f\Vert\ci{L^p(\mu)}^p,
\end{align*}
concluding the proof.
\end{proof}

Note that by applying the Monotone Convergence Theorem coupled with the weighted Littlewood--Paley estimates, Proposition \ref{prop:ParaproductTwoWeightEquiv} extends to all (not necessarily finitely supported) sequences.

Holmes--Petermichl--Wick \cite[Proposition~6.1]{holmes-petermichl-wick} prove that
\begin{equation*}
\Vert \Pi_{a}^{(1,1)}\Vert\ci{L^{p}(\mu)\rightarrow L^{p}(\lambda)}\lesssim\Vert a\Vert\ci{\text{BMO}\ci{\text{prod},\bfD}(\nu)},
\end{equation*}
where $\nu:=\mu^{1/p}\lambda^{-1/p}$ (which is an $A_2$ weight), $1 < p < \infty$ (in the general multiparameter case, the analogous result is due to Airta \cite{airta}).
From this and Proposition \ref{prop:ParaproductTwoWeightEquiv} it follows that for all $1 < p < \infty$ we have
\begin{equation*}
\Vert a\Vert\ci{\text{BMO}\ci{\text{prod},\bfD}(\mu,\lambda,p)} \lesssim \Vert a\Vert\ci{\text{BMO}\ci{\text{prod},\bfD}(\nu)}.
\end{equation*}
Since also $\nu=(\lambda')^{1/p'}(\mu')^{-1/p'}$, where $\mu':=\mu^{-1/(p-1)}$, $\lambda':=\lambda^{-1/(p-1)}$, we deduce as well
\begin{equation*}
\Vert a\Vert\ci{\text{BMO}\ci{\text{prod},\bfD}(\lambda',\mu',p')}\lesssim\Vert \Pi_{a}^{(1,1)}\Vert\ci{L^{p'}(\lambda')\rightarrow L^{p'}(\mu')}\lesssim\Vert a\Vert\ci{\text{BMO}\ci{\text{prod},\bfD}(\nu)}.
\end{equation*}
We show now that $\Vert \Pi_{a}^{(1,1)} \Vert\ci{L^p(\mu) \rightarrow L^p(\lambda)}$ and $\Vert a\Vert\ci{\text{BMO}\ci{\text{prod},\bfD}(\nu)}$ are actually equivalent, for all $1<p<\infty$.

\begin{prop}
\label{prop:ParaproductOneWeightEquiv}
Let $1 < p < \infty$ and dyadic biparameter $A_p$ weights $\mu$ and $\lambda$ on $\R\times\R$. Define $\nu := \mu^{1/p} \lambda^{-1/p}$. Then, there holds
\begin{equation*}
\Vert a\Vert\ci{\emph{BMO}\ci{\emph{prod},\bfD}(\nu)}\sim\Vert \Pi_{a}^{(1,1)}\Vert\ci{L^{p}(\mu)\rightarrow L^{p}(\lambda)}\sim\Vert \Pi_{a}^{(1,1)}\Vert\ci{L^{p'}(\lambda')\rightarrow L^{p'}(\mu')},
\end{equation*}
where the implied constants depend only on $p,[\mu]\ci{A_p,\bfD}$ and $[\lambda]\ci{A_p,\bfD}$.
\end{prop}

\begin{proof}
Recall that since $\mu$ is a dyadic biparameter $A_p$ weight, $\mu':=\mu^{-1/(p-1)}$ is a dyadic biparameter $A_{p'}$ weight with $[\mu']\ci{A_{p'},\bfD}=[\mu]\ci{A_{p},\bfD}^{p'-1}$. Similarly, $\lambda':=\lambda^{-1/(p-1)}$ is a dyadic biparameter $A_{p'}$ weight with $[\lambda']\ci{A_{p'},\bfD}=[\lambda]\ci{A_{p},\bfD}^{p'-1}$, and $\nu^{-1}$ is a dyadic biparameter $A_2$ weight  with $[\nu^{-1}]\ci{A_2,\bfD}=[\nu]\ci{A_2,\bfD}$. In particular, $\mu',\lambda',\nu^{-1}$ are also locally integrable. Using these observations and since the sequence $a$ is finitely supported, it is easy to see that
\begin{equation*}
\Vert a\Vert\ci{\text{BMO}\ci{\text{prod},\bfD}(\nu)},~\Vert \Pi_{a}^{(1,1)}\Vert\ci{L^{p}(\mu)\rightarrow L^{p}(\lambda)},~\Vert \Pi_{a}^{(1,1)}\Vert\ci{L^{p'}(\lambda')\rightarrow L^{p'}(\mu')}<\infty.
\end{equation*}
For brevity we set $C(a):=\Vert a\Vert\ci{\text{BMO}\ci{\text{prod},\bfD}(\nu)}$. By the weighted Littlewood--Paley estimates we have that
\begin{equation*}
\Vert \Pi_{a}^{(1,1)}\Vert\ci{L^{p}(\mu)\rightarrow L^{p}(\lambda)}\sim C_{1}(a),\qquad\Vert \Pi_{a}^{(1,1)}\Vert\ci{L^{p'}(\lambda')\rightarrow L^{p'}(\mu')}\sim C_{2}(a),
\end{equation*}
where $C_{1}(a), C_{2}(a)$ are the best finite nonnegative constants such that
\begin{equation*}
\left\Vert\left(\sum_{R\in\bfD}|a\ci{R}|^2\La |f|\Ra\ci{R}^2\frac{\1\ci{R}}{|R|}\right)^{1/2}\right\Vert\ci{L^{p}(\lambda)}\leq C_{1}(a)\Vert f\Vert\ci{L^{p}(\mu)},
\end{equation*}
\begin{equation*}
\left\Vert\left(\sum_{R\in\bfD}|a\ci{R}|^2\La |f|\Ra\ci{R}^2\frac{\1\ci{R}}{|R|}\right)^{1/2}\right\Vert\ci{L^{p'}(\mu')}\leq C_{2}(a)\Vert f\Vert\ci{L^{p'}(\lambda')},
\end{equation*}
for all measurable functions $f$ on $\R\times\R$. Thus, it suffices to prove that
\begin{equation}
\label{local_comparability}
C(a)\sim C_1(a)\sim C_2(a),
\end{equation}
where all implied constants depend on $p,[\mu]\ci{A_p,\bfD},[\lambda]\ci{A_p,\bfD}$. We will use \emph{bilinear} estimates. Fix any $\cU\subseteq\bfD$. Let $f,g$ be any two measurable functions on $\R\times\R$ taking nonnegative values, and consider the bilinear form
\begin{equation*}
B(f,g):=\sum_{R\in\cU}|a\ci{R}|^2\La f\Ra\ci{R}\La g\Ra\ci{R}\La\nu^{-1}\Ra\ci{R}.
\end{equation*}
Clearly $B(f,g)=\int_{\R\times\R}F(x)\nu^{-1}(x)dx$, where
\begin{equation*}
F:=\sum_{R\in\cU}|a\ci{R}|^2\La f\Ra\ci{R}\La g\Ra\ci{R}\frac{\1\ci{R}}{|R|}.
\end{equation*}
By Cauchy--Schwarz and H\ddoto lder's inequality we obtain
\begin{align*}
&B(f,g)=\int_{\R\times\R}F(x)\nu^{-1}(x)dx\\&
\leq\int_{\R\times\R}\left(\sum_{R\in\cU}|a\ci{R}|^2(\La f\Ra\ci{R})^2\frac{\1\ci{R}(x)}{|R|}\right)^{1/2}\left(\sum_{R\in\cU_n}|a\ci{R}|^2(\La g\Ra\ci{R})^2\frac{\1\ci{R}(x)}{|R|}\right)^{1/2}\nu^{-1}(x)dx\\
&=\int_{\R\times\R}\left(\sum_{R\in\cU}|a\ci{R}|^2(\La f\Ra\ci{R})^2\frac{\1\ci{R}(x)}{|R|}\right)^{1/2}\lambda^{1/p}(x)\left(\sum_{R\in\cU}|a\ci{R}|^2(\La g\Ra\ci{R})^2\frac{\1\ci{R}(x)}{|R|}\right)^{1/2}\mu^{-1/p}(x)dx\\
&\leq\left(\int_{\R\times\R}\left(\sum_{R\in\cU}|a\ci{R}|^2(\La f\Ra\ci{R})^2\frac{\1\ci{R}(x)}{|R|}\right)^{p/2}\lambda(x)dx\right)^{1/p}\\
&\cdot\left(\int_{\R\times\R}\left(\sum_{R\in\cU}|a\ci{R}|^2(\La g\Ra\ci{R})^2\frac{\1\ci{R}(x)}{|R|}\right)^{p'/2}\mu^{-p'/p}(x)dx\right)^{1/p'}\\
&\leq C_1(a)C_2(a)\Vert f\Vert\ci{L^{p}(\mu)}\Vert g\Vert\ci{L^{p'}(\lambda')}.
\end{align*}
We now pick
\begin{equation*}
f:=\mu^{-1/p}\nu^{1/p}\1\ci{\sh(\cU)},\qquad g:=\lambda^{1/p}\nu^{1/p'}\1\ci{\sh(\cU)}.
\end{equation*}
We have
\begin{equation*}
\Vert f\Vert\ci{L^{p}(\mu)}^{p}=\int_{\sh(\cU)}\mu^{-1}(x)\nu(x)\mu(x)dx=\nu(\sh(\cU)),
\end{equation*}
\begin{equation*}
\Vert g\Vert\ci{L^{p'}(\lambda')}^{p'}=\int_{\sh(\cU)}\lambda^{p'/p}(x)\nu(x)\lambda^{-1/(p-1)}(x)dx=\nu(\sh(\cU)),
\end{equation*}
therefore
\begin{equation*}
\Vert f\Vert\ci{L^{p}(\mu)}\Vert g\Vert\ci{L^{p'}(\lambda')}=\nu(\sh(\cU)).
\end{equation*}
So we have
\begin{equation*}
\sum_{R\in\cU}|a\ci{R}|^2\La \mu^{-1/p}\nu^{1/p}\Ra\ci{R}\La \lambda^{1/p}\nu^{1/p'}\Ra\ci{R}\La\nu^{-1}\Ra\ci{R}\leq C_1(a)C_2(a)\nu(\sh(\cU)).
\end{equation*}
Set $w:=\mu^{-1/p}\nu^{1/p}$. Then clearly
\begin{equation*}
\mu^{-1/p}\nu^{1/p}\lambda^{1/p}\nu^{1/p'}=1,
\end{equation*}
so an immediate application of Jensen's inequality (with exponent $-1<0$) gives
\begin{equation*}
\La \mu^{-1/p}\nu^{1/p}\Ra\ci{R}\La \lambda^{1/p}\nu^{1/p'}\Ra\ci{R}=\La w\Ra\ci{R}\La w^{-1}\Ra\ci{R}\geq 1.
\end{equation*}
Thus
\begin{equation*}
\sum_{R\in\cU}|a\ci{R}|^2\La\nu^{-1}\Ra\ci{R}\leq C_1(a)C_2(a)\nu(\sh(\cU)).
\end{equation*}
It follows that $C(a)^2\leq C_1(a)C_2(a)$. Since we already know that $C(a)\gtrsim C_i(a),~i=1,2$, we deduce $C(a)\sim C_1(a)\sim C_2(a)$.
\end{proof}

Note that by applying the Monotone Convergence Theorem coupled with the weighted Littlewood--Paley estimates, Proposition \ref{prop:ParaproductOneWeightEquiv} extends to all (not necessarily finitely supported) sequences.

Combining Proposition \ref{prop:ParaproductTwoWeightEquiv} and \ref{prop:ParaproductOneWeightEquiv}, we already deduce that if $1 < p \leq 2$, then
\begin{equation*}
\Vert a\Vert\ci{\text{BMO}\ci{\text{prod},\bfD}(\nu)}\sim\Vert a\Vert\ci{\text{BMO}\ci{\text{prod},\bfD}(\mu,\lambda,p)}\gtrsim\Vert a\Vert\ci{\text{BMO}\ci{\text{prod},\bfD}(\lambda',\mu',p')},
\end{equation*}
and that if $2 \leq p < \infty$, then
\begin{equation*}
\Vert a\Vert\ci{\text{BMO}\ci{\text{prod},\bfD}(\nu)}\sim\Vert a\Vert\ci{\text{BMO}\ci{\text{prod},\bfD}(\lambda',\mu',p')}\gtrsim\Vert a\Vert\ci{\text{BMO}\ci{\text{prod},\bfD}(\mu,\lambda,p)}.
\end{equation*}

The next proposition will complete the proof of Theorem \ref{t: equivBMO}.

\begin{prop}
\label{p: equivBMO-pgtr2}
Let $a,p,\mu,\lambda,\nu$ be as above. Assume $p>2$. Then, there holds
\begin{equation*}
\Vert a\Vert\ci{\emph{BMO}\ci{\emph{prod},\bfD}(\nu)}\lesssim \Vert a\Vert\ci{\emph{BMO}\ci{\emph{prod},\bfD}(\mu,\lambda,p)},
\end{equation*}
where the implied constant depends only on $p,[\mu]\ci{A_p,\bfD}$ and $[\lambda]\ci{A_p,\bfD}$.
\end{prop}

\begin{proof}
First of all, note that we have already proved that
\begin{equation*}
\Vert c\Vert\ci{\text{BMO}\ci{\text{prod},\bfD}(\nu,1,2)}\lesssim \Vert c\Vert\ci{\text{BMO}\ci{\text{prod},\bfD}(1,\nu^{-1},2)},
\end{equation*}
for \emph{any} sequence $c=\lbrace c\ci{R}\rbrace\ci{R\in\bfD}$. Applying this for the sequence $c=\lbrace c\ci{R}:=a\ci{R}\La \nu^{-1}\Ra\ci{R}^{1/2}\rbrace\ci{R\in\bfD}$, we deduce
\begin{align*}
\Vert a\Vert\ci{\text{BMO}\ci{\text{prod},\bfD}(\nu)}&=
\Vert \lbrace a\ci{R}\La\nu^{-1}\Ra\ci{R}^{1/2}\rbrace\ci{R\in\bfD}\Vert\ci{\text{BMO}\ci{\text{prod},\bfD}(\nu,1,2)}\lesssim
\Vert \lbrace a\ci{R}\La\nu^{-1}\Ra\ci{R}^{1/2}\rbrace\ci{R\in\bfD}\Vert\ci{\text{BMO}\ci{\text{prod},\bfD}(1,\nu^{-1},2)}\\
&=\Vert \lbrace a\ci{R}\La\nu^{-1}\Ra\ci{R}\rbrace\ci{R\in\bfD}\Vert\ci{\text{BMO}\ci{\text{prod},\bfD}}.
\end{align*}
Since $\mu$ is a dyadic biparameter $A_p$ weight and $p>2$, from the second half of the proof of \cite[Theorem 3.2]{martikainen et al} we have
\begin{equation*}
\Vert c\Vert\ci{\text{BMO}\ci{\text{prod},\bfD}}\lesssim\sup_{\cU\subseteq\bfD}\frac{1}{(\mu(\sh(\cU)))^{1/2}}\left\Vert\left(\sum_{R\in\cU}|c\ci{R}|^2\frac{\1\ci{R}}{|R|}\right)^{1/2}\right\Vert\ci{L^{2}(\mu)},
\end{equation*}
where the implied constant depends only on $p$ and $[\mu]\ci{A_p,\bfD}$, for any sequence $c=\lbrace c\ci{R}\rbrace\ci{R\in\bfD}$, so since $p>2$, by H\ddoto lder's inequality we deduce
\begin{equation*}
\Vert c\Vert\ci{\text{BMO}\ci{\text{prod},\bfD}}\lesssim\sup_{\cU\subseteq\bfD}\frac{1}{(\mu(\sh(\cU)))^{1/p}}\left\Vert\left(\sum_{R\in\cU}|c\ci{R}|^2\frac{\1\ci{R}}{|R|}\right)^{1/2}\right\Vert\ci{L^{p}(\mu)}.
\end{equation*}
Thus
\begin{align*}
\Vert \lbrace a\ci{R}\La\nu^{-1}\Ra\ci{R} \rbrace\ci{R\in\bfD}\Vert\ci{\text{BMO}\ci{\text{prod},\bfD}}\lesssim
\sup_{\cU\subseteq\bfD}\frac{1}{(\mu(\sh(\cU)))^{1/p}}\left\Vert\left(\sum_{R\in\cU}|a\ci{R}\La\nu^{-1}\Ra\ci{R}|^2\frac{\1\ci{R}}{|R|}\right)^{1/2}\right\Vert\ci{L^{p}(\mu)}.
\end{align*}
For all $\cU\subseteq\bfD$, using Lemma \ref{l: tl-upper bound} and Corollary \ref{c: tl-lower bound} (since $a$ is finitely supported) and the fact that $\La \nu^{-1}\Ra\ci{R}\sim\La\mu\Ra\ci{R}^{-1/p}\La\lambda\Ra\ci{R}^{1/p}$, for all $R\in\bfD$, we get
\begin{align*}
\left\Vert\left(\sum_{R\in\cU}|a\ci{R}\La\nu^{-1}\Ra\ci{R}|^2\frac{\1\ci{R}}{|R|}\right)^{1/2}\right\Vert\ci{L^{p}(\mu)}
&\sim \left\Vert\left(\sum_{R\in\cU}|a\ci{R}|^2\La\nu^{-1}\Ra\ci{R}^2\La\mu\Ra\ci{R}^{2/p}\frac{\1\ci{R}}{|R|}\right)^{1/2}\right\Vert\ci{L^{p}}\\
&\sim\left\Vert\left(\sum_{R\in\cU}|a\ci{R}|^2\La\lambda\Ra\ci{R}^{2/p}\frac{\1\ci{R}}{|R|}\right)^{1/2}\right\Vert\ci{L^{p}}\sim \left\Vert\left(\sum_{R\in\cU}|a\ci{R}|^2\frac{\1\ci{R}}{|R|}\right)^{1/2}\right\Vert\ci{L^{p}(\lambda)},
\end{align*}
concluding the proof.
\end{proof}

\section{Estimates for iterated commutators of Haar multipliers}
\label{sec:HaarMultipliers}

Let $\Sigma$ be the set of all finitely supported maps $\sigma:\cD\rightarrow\lbrace-1,0,1\rbrace$. In the sequel, the elements of $\Sigma$ similar spaces will be called \emph{sign choices}, and will always be considered to be finitely supported. For each $\sigma\in\Sigma$, we consider the \emph{Haar multiplier} $T_{\sigma},$ sometimes also called \emph{martingale transform}, on the real line given by
\begin{equation*}
T_{\sigma}f:=\sum_{I\in\cD}\sigma(I)f\ci{I}h\ci{I},\qquad f\in L^1\ti{loc}(\R).
\end{equation*}
Moreover, we consider Haar multipliers $T_{\sigma}^1,~T_{\sigma}^2$ acting on functions $f\in L^1\ti{loc}(\R^2)$ separately in each of the two variables, namely
\begin{equation*}
T_{\sigma}^1f(t,s):=T_{\sigma}(f(\fdot,s))(t),\qquad T_{\sigma}^2f(t,s):=T_{\sigma}(f(t,\fdot))(s),
\end{equation*}
for a.~e. $(t,s)\in\R^2$. It is clear that
\begin{equation*}
T_{\sigma}^{1}f=\sum_{I,J\in\cD}\sigma(I)f\ci{I\times J}h\ci{I\times J}=\sum_{I\in\cD}\sigma(I)Q^{1}\ci{I}f,
\end{equation*}
\begin{equation*}
 T_{\sigma}^{2}f=\sum_{I,J\in\cD}\sigma(J)f\ci{I\times J}h\ci{I\times J}=\sum_{J\in\cD}\sigma(J)Q^{2}\ci{J}f.
\end{equation*}

The main result of this section is Theorem \ref{t: HaarmultBMO}, which we recall here.

\theoremstyle{plain}
\newtheorem*{thm:HaarmultBMO*}{Theorem \ref{t: HaarmultBMO}}
\begin{thm:HaarmultBMO*}
Let $1 < p < \infty.$
Consider a function $b \in L^1\ti{loc}(\R^2),$ dyadic biparameter $A_p$ weights $\mu,$ $\lambda$ and define $\nu:=\mu^{1/p}\lambda^{-1/p}$. Then
\begin{equation*}
\sup_{\sigma_1,\sigma_2\in\Sigma}\Vert[T_{\sigma_1}^1,[T_{\sigma_2}^2,b]]\Vert\ci{L^p(\mu)\rightarrow L^p(\lambda)}\sim\Vert b\Vert\ci{\emph{BMO}\ci{\emph{prod},\bfD}(\nu)},
\end{equation*}
where the implied constants depend only on $p,$ $[\mu]\ci{A_p,\bfD}$ and $[\lambda]\ci{A_p,\bfD}$.
\end{thm:HaarmultBMO*}

In the unweighted case with $p = 2,$ Theorem \ref{t: HaarmultBMO} is proved by Blasco--Pott \cite{blasco-pott}. They get this result by averaging over the set of sign choices and then using orthogonality in Hilbert spaces. While we still rely on averaging over sign choices in our proof, we use a multiparameter extension of Khintchine's inequalities as well as vector-valued estimates to compensate for the lack of orthogonality for $p \neq 2.$

In the sequel we fix $1<p<\infty$, $b\in L^1\ti{loc}(\R^2)$, dyadic biparameter $A_p$ weights $\mu,\lambda$ on $\R^2$, and we set $\nu:=\mu^{1/p}\lambda^{-1/p}$. Note that we will be systematically suppressing from the notation dependence of implied constants on the value of $p$ and the Muckenhoupt characteristics $[\mu]\ci{A_p,\bfD}$ and $[\lambda]\ci{A_p,\bfD}$.

\subsection{Relating Haar multipliers to a ``symmetrized'' paraproduct}

Consider the ``purely non-cancellative'' and ``purely cancellative'' respectively biparameter paraproducts
\begin{equation*}
\Pi_{b}^{(1,1)}f:=\sum_{R\in\bfD}b\ci{R}\La f\Ra\ci{R}h\ci{R},\qquad \Pi_{b}^{(0,0)}f:=\sum_{R\in\bfD}b\ci{R}f\ci{R}\frac{\1\ci{R}}{|R|},
\end{equation*}
and the ``mixed non-cancellative--cancellative'' biparameter paraproducts
\begin{equation*}
\Pi_{b}^{(1,0)}f:=\sum_{R\in\bfD}b\ci{R}f^{(1,0)}\ci{R}h^{(0,1)}\ci{R},\qquad \Pi_{b}^{(0,1)}f:=\sum_{R\in\bfD}b\ci{R}f^{(0,1)}\ci{R}h^{(1,0)}\ci{R}
\end{equation*}
In the notation of each paraproduct, the superscript indicates the type of Haar functions against which the argument of the paraproduct is integrated, while the ``complementary'' pair of indices indicates the type of Haar functions appearing directly in the paraproduct.

Blasco--Pott \cite{blasco-pott} consider the ``symmetrized'' paraproduct in the biparameter setting
\begin{equation*}
\Lambda_{b}:=\Pi_{b}^{(11)}+\Pi_{b}^{(00)}+\Pi_{b}^{(10)}+\Pi_{b}^{(01)}.
\end{equation*}
Blasco--Pott \cite{blasco-pott} prove via direct computation that $\Lambda_{b}$ deserves to be called ``symmetrized'' paraproduct in the sense that
\begin{equation}
\label{equivsymparaprod}
\Lambda_{b}f=\sum_{R\in\bfD}(P\ci{R}b)f\ci{R}h\ci{R},
\end{equation}
where we recall that
\begin{equation*}
P\ci{R}b=\sum_{R'\in\bfD(R)}b\ci{R'}h\ci{R'}.
\end{equation*}
Note that in the general multiparameter setting of $\R^{\vec{d}}:=\R^{d_1}\times\dots\times\R^{d_t}$ one can also define an analogous operator $\boldsymbol{\Lambda}\ci{b}$ as a sum of generalized paraproducts in a way that one still has $\boldsymbol{\Lambda}_{b}f=\sum_{R\in\bfD}(\boldsymbol{P}\ci{R}b)f\ci{R}h\ci{R},$ where $\boldsymbol{P}\ci{R}$ denotes the multiparameter analogue to $P\ci{R}$ (and, abusing the notation, $\bfD$ denotes the set of all dyadic rectangles in the product space $\R^{\vec{d}}$).

It is important to note that for any $R=I\times J\in\bfD$, one has
\begin{equation}
\label{equivpartialproj}
P\ci{R}b(t,s)=(b-\La b(\fdot,s)\Ra\ci{I}-\La b(t,\fdot)\Ra\ci{J}+\La b\Ra\ci{I\times J})\1\ci{R}(t,s),
\end{equation}
for a.e.~$(t,s)\in\R^2$. Also, the same computation along the lines of the inclusion-exclusion principle that leads to \eqref{equivpartialproj} extends to yield a multiparameter analogue for $\boldsymbol{P}\ci{R}$. Using expressions \eqref{equivsymparaprod} and \eqref{equivpartialproj}, it is easy to see via direct computation, as remarked by Blasco--Pott \cite{blasco-pott}, that
\begin{equation}
\label{replace b by Lb-projections}
[Q\ci{I}^1,[Q\ci{J}^2,b]]=[Q\ci{I}^1,[Q\ci{J}^2,\Lambda_{b}]],\qquad\forall I,J\in\cD,
\end{equation}
and
\begin{equation}
\label{paraprod-proj}
Q^1\ci{I}\Lambda_{b}Q^1\ci{I}=0,\qquad Q^2\ci{J}\Lambda_{b}Q^2\ci{J}=0,\qquad\forall I,J\in\cD
\end{equation}
(and, as before, the analogues of these two expressions also hold in the multiparameter setting).
In fact, the weighted Littlewood--Paley estimates imply that the family of all finite linear combinations of Haar functions $h\ci{R},~R\in\bfD$ is dense in the weighted space $L^p(w)$, for any dyadic biparameter $A_p$ weight $w$, $1<p<\infty$, so one needs to check \eqref{replace b by Lb-projections} and \eqref{paraprod-proj} only on functions of this type. Note that \eqref{replace b by Lb-projections} immediately implies
\begin{equation}
\label{replace b by Lb-HaarMult}
[T_{\sigma_1}^{1},[T_{\sigma_2}^2,b]]=[T_{\sigma_1}^{1},[T_{\sigma_2}^{2},\Lambda_b]],\qquad\forall \sigma_1,\sigma_2\in\Sigma.
\end{equation}

The following lemma contains one of the most important steps towards the proof of  Theorem \ref{t: HaarmultBMO}.

\begin{lm}
\label{l: symparaprod-Haarmult-p}
There holds
\begin{equation*}
\Vert\Lambda_{b}\Vert\ci{L^{p}(\mu)\rightarrow L^{p}(\lambda)}\sim\sup_{\sigma_1,\sigma_2\in\Sigma}\Vert[T_{\sigma_1}^{1},[T_{\sigma_2}^{2},b]]\Vert\ci{L^{p}(\mu)\rightarrow L^{p}(\lambda)}.
\end{equation*}
\end{lm}

For the proof of Lemma \ref{l: symparaprod-Haarmult-p} we will rely on a straightforward extension of Khintchine's inequalities to the multiparameter setting, which is of independent interest.

\begin{lm}
\label{l: biparam_Khintchine}
Let $(\mathbb{X}_i,\bP_i)$, $i=1,2$ be probability spaces. For $j=1,2$, let $(X^{i}_{j})_{j=1}^{N_i}$ be a Rademacher sequence on $(\mathbb{X}_i,\bP_i)$, that is $X^{i}_{j},~j=1,\ldots,N_{j}$ are independent with
\begin{equation*}
\bP_i(X^{i}_{j}=1)=\bP_i(X^{i}_{j}=-1)=1/2,\qquad j=1,\ldots,N_j.
\end{equation*}
Let $A$ be an $N_1\times N_2$ (complex) matrix. Then, there holds
\begin{equation*}
\left\Vert\sum_{j_1=1}^{N_1}\sum_{j_2=1}^{N_2}A(j_1,j_2)X^1_{j_1}\otimes X^2_{j_2}\right\Vert\ci{L^q(\mathbb{X}_1\times\mathbb{X}_2)}\sim_{q,r}\left\Vert\sum_{j_1=1}^{N_1}\sum_{j_2=1}^{N_2}A(j_1,j_2)X^1_{j_1}\otimes X^2_{j_2}\right\Vert\ci{L^r(\mathbb{X}_1\times\mathbb{X}_2)}.
\end{equation*}
for all $0<q,r<\infty$. In particular
\begin{equation*}
\left\Vert\sum_{j_1=1}^{N_1}\sum_{j_2=1}^{N_2}A(j_1,j_2)X^1_{j_1}\otimes X^2_{j_2}\right\Vert\ci{L^q(\mathbb{X}_1\times\mathbb{X}_2)}\sim_{q}\left(\sum_{j_1=1}^{N_1}\sum_{j_2=1}^{N_2}|A(j_1,j_2)|^2\right)^{1/2},\qquad\forall 0<q<\infty.
\end{equation*}
\end{lm}

\begin{proof}

Let $0<q,r<\infty$ be arbitrary. Without loss of generality, we may assume $q<r$. Then, by H\ddoto lder's inequality, it suffices only to prove that
\begin{equation*}
\left\Vert\sum_{j_1=1}^{N_1}\sum_{j_2=1}^{N_2}A(j_1,j_2)X^1_{j_1}\otimes X^2_{j_2}\right\Vert\ci{L^r(\mathbb{X}_1\times\mathbb{X}_2)}\lesssim_{q,r}\left\Vert\sum_{j_1=1}^{N_1}\sum_{j_2=1}^{N_2}A(j_1,j_2)X^1_{j_1}\otimes X^2_{j_2}\right\Vert\ci{L^q(\mathbb{X}_1\times\mathbb{X}_2)}.
\end{equation*}
Set
\begin{equation*}
Y_{j_1}:=\sum_{j_2=1}^{N_2}A(j_1,j_2)X^2_{j_2},\qquad j_1=1,\ldots,N_1,
\end{equation*}
\begin{equation*}
Z_{j_2}:=\sum_{j_1=1}^{N_1}A(j_1,j_2)X^1_{j_1},\qquad j_2=1,\ldots,N_2.
\end{equation*}
Then, using first Khintchine's inequalities, then Minkowski's inequality (in view of the fact that $r/q\geq 1$), and finally again Khintchine's inequalities, we get
\begin{align*}
&\left\Vert\sum_{j_1=1}^{N_1}\sum_{j_2=1}^{N_2}A(j_1,j_2)X^1_{j_1}\otimes X^2_{j_2}\right\Vert\ci{L^r(\mathbb{X}_1\times\mathbb{X}_2)}^{r}=\int_{\mathbb{X}_2}\left(\int_{\mathbb{X}_1}\left|\sum_{j_1=1}^{N_1}Y_{j_1}(\omega_2)X_{j_1}^1(\omega_1)\right|^{r}d\bP_1(\omega_1)\right)d\bP_2(\omega_2)\\
&\sim_{q,r}\int_{\mathbb{X}_2}\left(\int_{\mathbb{X}_1}\left|\sum_{j_1=1}^{N_1}Y_{j_1}(\omega_2)X_{j_1}^1(\omega_1)\right|^{q}d\bP_1(\omega_1)\right)^{r/q}d\bP_2(\omega_2)\\
&\leq\left(\int_{\mathbb{X}_1}\left(\int_{\mathbb{X}_2}\left|\sum_{j_1=1}^{N_1}Y_{j_1}(\omega_2)X_{j_1}^1(\omega_1)\right|^{q\cdot\frac{r}{q}}d\bP_2(\omega_2)\right)^{q/r}d\bP_1(\omega_1)\right)^{r/q}\\
&=\left(\int_{\mathbb{X}_1}\left(\int_{\mathbb{X}_2}\left|\sum_{j_2=1}^{N_2}Z_{j_2}(\omega_1)X_{j_2}^2(\omega_2)\right|^{r}d\bP_2(\omega_2)\right)^{q/r}d\bP_1(\omega_1)\right)^{r/q}\\
&\sim_{q,r}\left(\int_{\mathbb{X}_1}\left(\int_{\mathbb{X}_2}\left|\sum_{j_2=1}^{N_2}Z_{j_2}(\omega_1)X_{j_2}^2(\omega_2)\right|^{q}d\bP_2(\omega_2)\right)d\bP_1(\omega_1)\right)^{r/q}\\
&=\left\Vert\sum_{j_1=1}^{N_1}\sum_{j_2=1}^{N_2}A(j_1,j_2)X^1_{j_1}\otimes X^2_{j_2}\right\Vert\ci{L^q(\mathbb{X}_1\times\mathbb{X}_2)}^{r},
\end{align*}
concluding the proof.

The second claim follows immediately from the first by just noting that an iteration of independence gives
\begin{equation*}
\left\Vert\sum_{j_1=1}^{N_1}\sum_{j_2=1}^{N_2}A(j_1,j_2)X^1_{j_1}\otimes X^2_{j_2}\right\Vert\ci{L^2(\mathbb{X}_1\times\mathbb{X}_2)}=\left(\sum_{j_1=1}^{N_1}\sum_{j_2=1}^{N_2}|A(j_1,j_2)|^2\right)^{1/2}.
\end{equation*}
\end{proof}

Clearly, one can use induction to prove similarly a multiparameter version of Khintchine's inequalities, for any $0<q,r<\infty$, in any number of parameters (one has just to replace any of the two uses of Khintchine's inequalities in the proof above by use of the inductive hypothesis).

\begin{proof}[Proof (of Lemma \ref{l: symparaprod-Haarmult-p})]
We first consider the direction $\gtrsim$. It is well-known that Haar multipliers on $\R$ are bounded from $L^p(w)$ into $L^p(w)$ for dyadic $A_p$ weights $w$ on $\R$, within constants depending only on $[w]\ci{A_p,\cD}$ (and not the sign choice in the definition of the Haar multiplier). It follows immediately that for any biparameter dyadic $A_p$ weight $w$ on $\R\times\R$ there holds
\begin{equation*}
\Vert T_{\sigma}^i\Vert\ci{L^p(w)\rightarrow L^p(w)}\lesssim\ci{[w]\ci{A_p,\bfD}}1,\qquad\forall i=1,2,\qquad\forall \sigma\in\Sigma,
\end{equation*}
therefore
\begin{align*}
&\Vert[T_{\sigma_1}^{1},[T_{\sigma_2}^{2},b]]\Vert\ci{L^p(\mu)\rightarrow L^p(\lambda)}=\Vert [T_{\sigma_1}^{1},[T_{\sigma_2}^{2},\Lambda_{b}]]\Vert\ci{L^p(\mu)\rightarrow L^p(\lambda)}\\
&=
\Vert T_{\sigma_1}^{1}T_{\sigma_2}^{2}\Lambda_{b}-T_{\sigma_1}^{1}\Lambda_{b}T_{\sigma_2}^{2}-T_{\sigma_2}^{2}\Lambda_{b}T_{\sigma_1}^{1}+\Lambda_{b}T_{\sigma_2}^{2}T_{\sigma_1}^{1}\Vert\ci{L^p(\mu)\rightarrow L^p(\lambda)}\lesssim 4\Vert \Lambda_{b}\Vert\ci{L^p(\mu)\rightarrow L^p(\lambda)},
\end{align*}
for all $\sigma_1,\sigma_2\in\Sigma$.

We now turn to the other direction. Let $(\cF_n)^{\infty}_{n=1}$ be an increasing sequence of subsets of $\cD$ exhausting $\cD$. For each $n=1,2,\ldots$, let $\Sigma_n$ be the set of all maps $\sigma:\cD\rightarrow\lbrace-1,0,1\rbrace$ that vanish outside of $\cF_n$ and that take values only $-1,1$ on $\cF_n$, and consider the natural probability measure $\bP_n$ on $\Sigma_n$ that to each coordinate $I\in\cF_n$ assigns each of the values $1$ and $-1$ with probability $1/2$, independently of all the other coordinates. Clearly, it suffices to prove that
\begin{equation}
\label{main goal-p}
\sup_{n=1,2,\ldots}\int_{\Sigma_n\times\Sigma_n}\Vert [T_{\sigma_1}^1,[T_{\sigma_2}^2,b]](f)\Vert\ci{L^p(\lambda)}^pd(\bP_n\otimes\bP_n)(\sigma_1,\sigma_{2})\gtrsim\Vert \Lambda_{b}(f)\Vert\ci{L^p(\lambda)}^p,
\end{equation}
for all (suitable) functions $f$ on $\R^2$. For brevity we set $\cC\ci{I\times J}:=[Q^1\ci{I},[Q^2\ci{J},b]]$. Applying Lemma \ref{l: biparam_Khintchine}, we have
\begin{align*}
&\sup_{n=1,2,\ldots}\int_{\Sigma_n\times\Sigma_n}\Vert [T_{\sigma_1}^1,[T_{\sigma_2}^2,b]](f)\Vert\ci{L^p(\lambda)}^pd(\bP_n\otimes\bP_n)(\sigma_1,\sigma_{2})\\
&=\sup_{n=1,2,\ldots}\int_{\Sigma_n\times\Sigma_n}\left\Vert\sum_{I\times J\in\bfD}\sigma_1(I)\sigma_2(J)\cC\ci{I\times J}(f)\right\Vert\ci{L^p(\lambda)}^pd(\bP_n\otimes\bP_n)(\sigma_1,\sigma_2)\\
&=\sup_{n=1,2,\ldots}\int_{\R\times\R}\left(\int_{\Sigma_n\times\Sigma_n}\left|\sum_{I\times J\in\bfD}\sigma_1(I)\sigma_2(J)\cC\ci{I\times J}(f)(x)\right|^pd(\bP_n\otimes\bP_n)(\sigma_1,\sigma_2)\right)\lambda(x)dx\\
&\sim_{p}\sup_{n=1,2,\ldots}\int_{\R\times\R}\left(\sum_{I,J\in\cF_n}|\cC\ci{I\times J}(f)(x)|^2\right)^{p/2}\lambda(x)dx=\int_{\R\times\R}\left(\sum_{I\times J\in\bfD}|\cC\ci{I\times J}(f)(x)|^2\right)^{p/2}\lambda(x)dx,
\end{align*}
where in the last equality we applied the Monotone Convergence Theorem. 
Observe that
\begin{equation*}
|Q\ci{R}g|\leq\La |g|\Ra\ci{R}\1\ci{R}\leq M\ci{\bfD}g,\qquad\forall R\in\bfD.
\end{equation*}
Moreover, O. N. Capri and C. Guti\'errez \cite{capri-gutierrez} establish the following one-weight vector-valued estimate for the dyadic strong maximal function (in any number of parameters):
\begin{equation*}
\left\Vert\left(\sum_{n=1}^{\infty}|M\ci{\bfD}g_n|^2\right)^{1/2}\right\Vert\ci{L^{p}(w)}\lesssim \left\Vert\left(\sum_{n=1}^{\infty}|g_n|^2\right)^{1/2}\right\Vert\ci{L^{p}(w)},
\end{equation*}
where the implied constants depend only on $p$ and $[w]\ci{A_p,\bfD}$ (their proof is for the case of the strong maximal function $M\ti{S}$ and multiparameter $A_p$ weights $w$, but it works without any changes for the case of the dyadic strong maximal function $M\ci{\bfD}$ and dyadic multiparameter $A_p$ weights $w$). Thus, we have
\begin{align*}
&\int_{\R\times\R}\left(\sum_{I\times J\in\bfD}|\cC\ci{I\times J}(f)(x)|^2\right)^{p/2}\lambda(x)dx\gtrsim\int_{\R\times\R}\left(\sum_{I\times J\in\bfD}|M\ci{\bfD}(\cC\ci{I\times J}(f))(x)|^2\right)^{p/2}\lambda(x)dx\\
&\geq \int_{\R\times\R}\left(\sum_{I\times J\in\bfD}|Q\ci{I}^1Q\ci{J}^2(\cC\ci{I\times J}(f))(x)|^2\right)^{p/2}\lambda(x)dx\\
&=\int_{\R\times\R}\left(\sum_{I\times J\in\bfD}|Q\ci{I}^1Q\ci{J}^2([Q^1\ci{I},[Q^2\ci{J},\Lambda_b]](f))(x)|^2\right)^{p/2}\lambda(x)dx,
\end{align*}
where in the last equality we have used \eqref{replace b by Lb-projections}. Notice that using \eqref{paraprod-proj} we obtain
\begin{align*}
&Q\ci{I}^1Q\ci{J}^2([Q^1\ci{I},[Q^2\ci{J},\Lambda_b]](f))\\
&=Q^1\ci{I}Q^2\ci{J}(Q^1\ci{I}Q^2\ci{J}\Lambda_{b}(f)-Q^1\ci{I}\Lambda_{b}Q\ci{J}^2(f)-Q^2\ci{J}\Lambda_{b}Q^1\ci{I}(f)+\Lambda_{b}Q\ci{I}^{1}Q\ci{J}^2(f))\\
&=Q^1\ci{I}Q^2\ci{J}\Lambda_{b}(f),
\end{align*}
for all $I,J\in\cD$. It follows that
\begin{align*}
&\int_{\R\times\R}\left(\sum_{I\times J\in\bfD}|Q\ci{I}^1Q\ci{J}^2([Q^1\ci{I},[Q^2\ci{J},\Lambda_b]](f))(x)|^2\right)^{p/2}\lambda(x)dx\\
&= \int_{\R\times\R}\left(\sum_{I\times J\in\bfD}|Q^1\ci{I}Q^2\ci{J}\Lambda_{b}(f)(x)|^2\right)^{p/2}\lambda(x)dx\sim \Vert \Lambda_{b}(f)\Vert\ci{L^p(\lambda)}^p,
\end{align*}
concluding the proof.
\end{proof}

\subsection{Bounds for the ``symmetrized" paraproduct and conclusion of the proof} In this section we complete the proof of Theorem \ref{t: HaarmultBMO}.
Blasco--Pott \cite{blasco-pott} show that
\begin{equation*}
P\ci{\Omega}(b) = P\ci{\Omega}(\Lambda_{b}(\1\ci{\Omega})).
\end{equation*}
This can be readily checked by direct computation using the definition of the operator $\Lambda\ci{b}$ and how paraproducts act on characteristic functions.
From this, it follows that
\begin{align*}
\Vert P\ci{\Omega}(b)\Vert\ci{L^p(\lambda)}=\Vert P\ci{\Omega}(\Lambda_{b}(\1\ci{\Omega}))\Vert\ci{L^p(\lambda)}\lesssim \Vert \Lambda_{b}(\1\ci{\Omega})\Vert\ci{L^p(\lambda)}\leq \Vert\Lambda_{b}\Vert\ci{L^p(\mu)\rightarrow L^p(\lambda)}(\mu(\Omega))^{1/p}.
\end{align*}
In the $\lesssim$ we used the weighted Littlewood--Paley estimates.
The analogous expressions are also valid for the multiparameter operators $\boldsymbol{P}\ci{\Om}$ and $\boldsymbol{\Lambda}\ci{b}$ and measurable sets $\Om \subseteq \R^{\vec{d}},$ and their proofs use the same idea of checking the action of the various paraproducts on characteristic functions.
It follows that
\begin{equation*}
\Vert b\Vert\ci{\text{BMO}\ci{\text{prod},\bfD}(\mu,\lambda,p)}\lesssim\Vert\Lambda_{b}\Vert\ci{L^p(\mu)\rightarrow L^p(\lambda)}.
\end{equation*}
Combining this with Theorem \ref{t: equivBMO} and Lemma \ref{l: symparaprod-Haarmult-p} we deduce
\begin{equation}
\label{upper_bound}
\Vert b\Vert\ci{\text{BMO}\ci{\text{prod},\bfD}(\nu)}\lesssim\sup_{\sigma_1,\sigma_2\in\Sigma}\Vert[T_{\sigma_1}^{1},[T_{\sigma_2}^{2},b]]\Vert\ci{L^{p}(\mu)\rightarrow L^{p}(\lambda)}.
\end{equation}
Moreover, Holmes--Petermichl--Wick \cite{holmes-petermichl-wick} prove that
\begin{equation*}
\Vert P_{b}\Vert\ci{L^p(\mu)\rightarrow L^p(\lambda)}\lesssim\Vert b\Vert\ci{\text{BMO}\ci{\text{prod},\bfD}(\nu)},
\end{equation*}
where $P_b$ is any of the four paraproducts $\Pi_b^{(\e_1\e_2)}$, $\e_1,\e_2\in\lbrace0,1\rbrace$ (the same fact for all relevant multiparameter paraproducts is shown by Airta in \cite{airta}). It follows that
\begin{equation}
\label{symparaprod-BMO-p}
\Vert\Lambda_{b}\Vert\ci{L^p(\mu)\rightarrow L^p(\lambda)}\lesssim\Vert b\Vert\ci{\text{BMO}\ci{\text{prod},\bfD}(\nu)}.
\end{equation}
Thus, combining Lemma \ref{l: symparaprod-Haarmult-p}, \eqref{upper_bound} and \eqref{symparaprod-BMO-p} we deduce
\begin{equation*}
\Vert b\Vert\ci{\text{BMO}\ci{\text{prod},\bfD}(\nu)}\sim\sup_{\sigma_1,\sigma_2\in\Sigma}\Vert[T_{\sigma_1}^{1},[T_{\sigma_2}^{2},b]]\Vert\ci{L^{p}(\mu)\rightarrow L^{p}(\lambda)},
\end{equation*}
concluding the proof.

Note that since $\nu$ is a dyadic biparameter $A_2$ weight on $\R^2$, $\nu=\nu^{1/2}(\nu^{-1})^{-1/2}$ and $[\nu]\ci{A_2,\bfD}=[\nu^{-1}]\ci{A_2,\bfD}$, we also get
\begin{equation*}
\Vert b\Vert\ci{\text{BMO}\ci{\text{prod},\bfD}(\nu)}\sim\sup_{\sigma_1,\sigma_2\in\Sigma}\Vert [T_{\sigma_1}^1,[T_{\sigma_2}^2,b]]\Vert\ci{L^2(\nu)\rightarrow L^2(\nu^{-1})},
\end{equation*}
where all implied constants depend only on $[\nu]\ci{A_2,\bfD}$.

\subsection{General multiparameter result}
The ideas presented in this section can also be applied, with only minor modifications, to any multiparameter setting. That is, Theorem \ref{t: HaarmultBMO} can be stated and proved for iterated commutators on functions defined on $\R^{\vec{d}} := \R^{d_1}\times\dots\times\R^{d_t},$ $\vec{d}:=(d_1,\ldots,d_t).$
We have already been commenting along the proof which steps have to be modified in this context.
Abusing slightly the notation, here we use $\bfD$ to denote the set of dyadic rectangles in the product space $\R^{\vec{d}}.$
The statement of the result in full generality is the following.

\begin{thm}
\label{t:GeneralHaarmultBMO}
Let $1 < p < \infty.$ Consider a function $b \in L^1\ti{loc}(\R^{\vec{d}}),$ dyadic multiparameter $A_p$ weights $\mu,$ $\lambda$ on $\R^{\vec{d}}$ and define $\nu:=\mu^{1/p}\lambda^{-1/p}$.
Then
\begin{equation*}
\sup_{\sigma_1,\ldots,\sigma_t\in\Sigma}\Vert[T_{\sigma_1}^1,[\ldots[T_{\sigma_t}^t,b]\ldots]]\Vert\ci{L^p(\mu)\rightarrow L^p(\lambda)}\sim\Vert b\Vert\ci{\emph{BMO}\ci{\emph{prod},\bfD}(\nu)},
\end{equation*}
where the implied constants depend only on $\vec{d},p,$ $[\mu]\ci{A_p,\bfD}$ and $[\lambda]\ci{A_p,\bfD}$.
\end{thm}

\section{Bounds for general commutators of Haar multipliers}
\label{sec:IndexedSpaces}

In this section, we show bounds analogous to those of Section \ref{sec:HaarMultipliers} for iterated commutators of general martingale transforms, not necessarily of tensor type.
As usual, consider dyadic multiparameter $A_p$  weights $\mu$ and $\lambda$ on $\R^{d}$, with $1 < p < \infty,$ and let $\nu := \mu^{1/p} \lambda^{-1/p}.$
We define the \emph{dyadic two-weight little BMO p-norm} $\Vert b \Vert\ci{\text{bmo}\ci{\bfD}(\mu,\lambda,p)}$ by
\begin{equation*}
  \Vert b \Vert\ci{\text{bmo}\ci{\bfD}(\mu,\lambda,p)} :=
  \sup_{R \in \bfD} \frac{1}{(\mu(R))^{1/p}} \Vert (b-\La b \Ra\ci{R}) \1\ci{R} \Vert\ci{L^p(\lambda)}.
\end{equation*}
We also say, for a dyadic multiparameter $A_2$ weight $w$ on $\R^d$, that a function $b$ is in the \emph{dyadic one-weight little BMO} space if
\begin{equation*}
  \sup_{R \in \bfD} \frac{1}{w(R)} \Vert (b-\La b \Ra\ci{R}) \1\ci{R} \Vert\ci{L^1(\R^d)} < \infty.
\end{equation*}
In this case, we assign to function $b$ the norm
\begin{equation*}
  \Vert b \Vert\ci{\text{bmo}\ci{\bfD}(w)} := \Vert b \Vert\ci{\text{bmo}\ci{\bfD}(w,w^{-1},2)},
\end{equation*}
which is equivalent to the previous supremum.
This equivalence is shown by Holmes--Petermichl--Wick \cite{holmes-petermichl-wick} in the biparameter case using an iteration of the one-parameter argument due to Holmes--Lacey--Wick \cite{holmes-lacey-wick}, but the same argument can be iterated to any number of parameters.
Furthermore, the authors of \cite{holmes-petermichl-wick} also prove that, for $\mu,$ $\lambda$ and $\nu$ as before, there holds
\begin{equation*}
  \Vert b \Vert\ci{\text{bmo}\ci{\bfD}(\mu,\lambda,p)} \sim \Vert b \Vert\ci{\text{bmo}\ci{\bfD}(\lambda',\mu',p')} \sim \Vert b \Vert\ci{\text{bmo}\ci{\bfD}(\nu)}
\end{equation*}
(it is actually shown there for the continuous biparameter setting, but their argument holds equally well in the dyadic case and can be iterated to any number of parameters as well).

Recall that if $x \in \R^{\vec{d}}$ and $k \in \lbrace 1, 2, \ldots, t\rbrace,$ we denote $x_{\overline{k}} := (x_1,\ldots,x_{k-1},\cdot,x_{k+1},\ldots,x_t),$ so that a function $f(x_{\overline{k}})$ depends only on the variable $x_k.$
Similarly, we also denote $\R^{\vec{d}_{\overline{k}}} := \R^{d_1} \times \dots \times \R^{d_{k-1}} \times \R^{d_{k+1}} \times \dots \times \R^{d_{t}}.$
In general, we extend this notation to any number of parameters, so that if $k_1,\ldots,k_s \in \lbrace 1, 2, \ldots, t\rbrace,$ then we take $x_{\overline{k_1,\ldots,k_s}} := (x_1,\ldots,x_{k_1-1},\cdot,x_{k_1+1},\ldots,x_{k_s-1},\cdot,x_{k_s+1},\ldots,x_t),$ and similarly for $\R^{\vec{d}_{\overline{k_1,\ldots,k_s}}}.$
Using this notation, Holmes--Petermichl--Wick \cite{holmes-petermichl-wick} show that
\begin{equation*}
  \Vert b \Vert\ci{\text{bmo}\ci{\bfD}(\nu)} \sim
  \max \bigg\lbrace \esssup_{x_{\overline{k}}\in\R^{\vec{d}_{\overline{k}}}} \Vert b(x_{\overline{k}}) \Vert\ci{\text{BMO}\ci{\bfD}(\nu(x_{\overline{k}}))}\colon k \in \lbrace 1,2,\ldots,t\rbrace\bigg\rbrace,
\end{equation*}
where $\Vert \cdot \Vert\ci{\text{BMO}\ci{\cD}(\nu)}$ denotes the usual dyadic weighted one-parameter BMO norm (again, the result is stated and proved there only in the continuous biparameter setting, but the argument holds as well for the dyadic spaces and in any number of parameters). In other words, a function $b$ is in dyadic weighted little bmo if and only if it is uniformly in dyadic weighted one-parameter BMO in each variable. Moreover, they also observe that
\begin{equation*}
 \text{bmo}\ci{\bfD}(\nu) \subseteq \BMOprodD(\nu).
\end{equation*}

More generally, let $\cI=\lbrace I_1,\ldots,I_l\rbrace$ be any partition of $\lbrace1,\ldots,t\rbrace$. Let $\mu,\lambda$ be $t$-parameter $A_p$ weights on $\R^{\vec{d}}$, $1<p<\infty$. Define
\begin{equation*}
\Vert b\Vert\ci{\text{bmo}\ci{\bfD}^{\cI}(\mu,\lambda,p)}:=
\max_{\imath\in I_{1}\times\dots\times I_{l}}\bigg(\esssup_{x_{\bar{\imath}}\in\R^{\vec{d}_{\bar{\imath}}}}\Vert b(x_{\bar{\imath}})\Vert\ci{\text{BMO}\ci{\text{prod}\ci{\bfD}(\mu(x_{\bar{\imath}}),\lambda(x_{\bar{\imath}}),p)}}\bigg),
\end{equation*}
and as before if $w$ is any $t$-parameter $A_2$ weight on $\R^{\vec{d}}$ then we define
\begin{equation*}
\Vert b\Vert\ci{\text{bmo}\ci{\bfD}^{\cI}(w)}:=\Vert b\Vert\ci{\text{bmo}\ci{\bfD}^{\cI}(w,w^{-1},2)}.
\end{equation*}
If $\cI$ consists only of singletons, then one recovers product BMO in $t$ parameters, while if $\cI$ has just one element, then one recovers little BMO.
In general, for any partition $\cI$ we have that
\begin{equation*}
  \text{bmo}\ci{\bfD}(w) \subseteq \text{bmo}\ci{\bfD}^{\cI}(w) \subseteq \BMOprodD(w).
\end{equation*}
Observe that our results in Section \ref{sec:WeightedEquivalences} immediately imply that if $\nu:=\mu^{1/p}\lambda^{-1/p}$, then
\begin{equation*}
\Vert b\Vert\ci{\text{bmo}^{\cI}\ci{\bfD}(\nu)}\sim\Vert b\Vert\ci{\text{bmo}^{\cI}\ci{\bfD}(\mu,\lambda,p)},
\end{equation*}
where all implied constants depend only on $p$, $[\mu]\ci{A_{p},\bfD}$ and $[\lambda]\ci{A_p,\bfD}$.
Notice that here we are making use of the fact that for any $\imath \in I_1\times \dots \times I_l$, and for almost every $x_{\overline{\imath}}\in\R^{\vec{d}_{\overline{\imath}}}$, the weight $\mu(x_{\overline{\imath}})$ is a dyadic $l$-parameter $A_p$ weight on $\R^{d_{i_1}}\times \dots \times \R^{d_{i_l}}$ with $[\mu(x_{\overline{\imath}})]\ci{A_p,\bfD}\leq[\mu]\ci{A_p,\bfD}$, and similarly for $\lambda$ and $\nu$.

Consider now the set $\bsigma$ of sign choices $\sigma=\lbrace\sigma(R)\rbrace\ci{R\in\bfD}$. For each $\sigma \in \bsigma$ we consider the Haar multiplier $T_\sigma$ defined by
\begin{equation*}
  T_\sigma f \colon= \sum_{R \in \bfD} \sigma(R) f\ci{R} h\ci{R}, \qquad f \in L^2(\R^{\vec{d}}).
\end{equation*}
Observe that here we consider all possible choices of signs over the dyadic rectangles in $\bfD,$ not only those that are of tensor type, as opposed to Section \ref{sec:HaarMultipliers}.
We will study bounds for $\Vert [T_\sigma,b] \Vert\ci{L^p(\mu) \rightarrow L^p(\lambda)}$ in terms of $\Vert b \Vert\ci{\text{bmo}\ci{\bfD}(\nu)}.$ We adapt the arguments presented in Section \ref{sec:HaarMultipliers} to the biparameter case.

\begin{prop}
Let $1 < p < \infty.$
Consider a function $b \in L^1\ti{loc}(\R^2),$ dyadic biparameter $A_p$ weights $\mu,$ $\lambda$ on $\R^2$ and define $\nu:=\mu^{1/p}\lambda^{-1/p}$. Then
\begin{equation*}
\sup_{\sigma \in \bsigma}\Vert[T_{\sigma},b]\Vert\ci{L^p(\mu)\rightarrow L^p(\lambda)} \sim \Vert b\Vert\ci{\emph{bmo}\ci{\bfD}(\nu)},
\end{equation*}
where the implied constants depend only on $p,$ $[\mu]\ci{A_p,\bfD}$ and $[\lambda]\ci{A_p,\bfD}$.
\end{prop}

\begin{proof}
Airta \cite[Theorem~4.12]{airta} shows general upper bounds for iterated commutators of multi-parameter Haar shifts in terms of the symbol norm in the appropriate indexed BMO space.
In particular, for a single commutator, this includes the upper bound $\Vert[T_{\sigma},b]\Vert\ci{L^p(\mu)\rightarrow L^p(\lambda)} \lesssim \Vert b\Vert\ci{\text{bmo}\ci{\bfD}(\nu)}.$
Thus, we only need to show the corresponding lower bound.

Define the operator $\Theta_b$ by
\begin{equation*}
  \Theta_b f = \sum_{R \in \bfD} (\omega\ci{R} b) f\ci{R} h\ci{R}, \qquad f \in L^2(\R^2),
\end{equation*}
where $\omega\ci{R} b := (b(x) - \La b \Ra\ci{R}) \1\ci{R}(x).$
This operator satisfies the relations
\begin{equation*}
  [Q\ci{R},b] = [Q\ci{R},\Theta_b], \qquad \forall R \in \bfD,
\end{equation*}
and
\begin{equation*}
  Q\ci{R} \Theta_b Q\ci{R} = 0, \qquad \forall R \in \bfD,
\end{equation*}
where $Q_R$ denotes the orthogonal projection from $L^2(\R^2)$ onto the one-dimensional space spanned by $h\ci{R}.$
Observe that a simple computation using the definition of $\Theta_b$ shows that they hold for any Haar function $h\ci{R},$ from which the general result follows by a density argument.

Now we see that
\begin{equation}
  \label{eq:ThetaNormEquiv}
  \sup_{\sigma \in \bsigma} \Vert [T_\sigma,b] \Vert\ci{L^p(\mu) \rightarrow L^p(\lambda)}
  \gtrsim \Vert \Theta_b \Vert\ci{L^p(\mu) \rightarrow L^p(\lambda)}.
\end{equation}
As before, the lower bound \eqref{eq:ThetaNormEquiv} will follow from that of the average of the commutator norms over $\bsigma.$
In this case we have
\begin{align*}
&\int_{\bsigma}\Vert [T_{\sigma},b](f)\Vert\ci{L^p(\lambda)}^p\, d\bP(\sigma)\\
&=\int_{\R^2}\left(\int_{\bsigma}\left|\sum_{R \in \bfD} \sigma(R)[Q\ci{R},b](f)(x)\right|^p\, d\bP(\sigma)\right)\lambda(x)dx\\
&\sim\int_{\R^2}\left(\int_{\bsigma}\left|\sum_{R \in \bfD}\sigma(R)[Q\ci{R},b](f)(x)\right|^2\, d\bP(\sigma)\right)^{p/2}\lambda(x)dx,
\end{align*}
where we have used Khintchine's inequalities in the last step.
This last quantity is equal to
\begin{align*}
&=\int_{\R^2}\left(\sum_{R \in\bfD}|[Q\ci{R},b](f)(x)|^2\right)^{p/2}\lambda(x)dx\\
&\gtrsim\int_{\R^2}\left(\sum_{R \in\bfD}|M\ci{\bfD}[Q\ci{R},b](f)(x)|^2\right)^{p/2}\lambda(x)dx\\
&\geq\int_{\R^2}\left(\sum_{R \in\bfD}|Q\ci{R}[Q\ci{R},\Theta_b](f)(x)|^2\right)^{p/2}\lambda(x)dx\\
&=\int_{\R^2}\left(\sum_{R \in\bfD}|Q\ci{R}\Theta_b(f)(x)|^2\right)^{p/2}\lambda(x)dx \sim \Vert \Theta_b(f) \Vert\ci{L^p(\lambda)}^p.
\end{align*}
This shows \eqref{eq:ThetaNormEquiv}.

Now we are only left with checking that
\begin{equation*}
  \Vert \Theta_b \Vert\ci{L^p(\mu) \rightarrow L^p(\lambda)} \gtrsim
  \Vert b\Vert\ci{\text{bmo}\ci{\bfD}(\nu)}.
\end{equation*}
Observe that testing the operator on Haar functions we immediately get
\begin{equation*}
  \Vert \Theta_b \Vert\ci{L^p(\mu) \rightarrow L^p(\lambda)} \Vert h\ci{R} \Vert\ci{L^p(\mu)} \geq \Vert \Theta_b h\ci{R} \Vert\ci{L^p(\lambda)}
  = \Vert (b - \La b \Ra\ci{R}) \1\ci{R} \Vert\ci{L^p(\lambda)} |R|^{-1/2}.
\end{equation*}
We also have that $\Vert h\ci{R} \Vert\ci{L^p(\mu)} = (\mu(R))^{1/p} |R|^{-1/2},$ so that
\begin{equation*}
  \Vert \Theta_b \Vert\ci{L^p(\mu) \rightarrow L^p(\lambda)} \geq \frac{1}{(\mu(R))^{1/p}} \Vert (b - \La b \Ra\ci{R}) \1\ci{R} \Vert\ci{L^p(\lambda)}.
\end{equation*}
But the supremum of the right-hand side is precisely $\Vert b\Vert\ci{\text{bmo}\ci{\bfD}(\mu,\lambda,p)} \sim \Vert b\Vert\ci{\text{bmo}\ci{\bfD}(\nu)}.$
\end{proof}

This method can also be adapted to general indexed BMO spaces, and thus to the general multiparameter little bmo case.
We explain how to do it taking as an example the product space $\R^3 = \R\times\R\times\R,$ and the partition $\cI:=\lbrace I_1:=\lbrace1,3\rbrace,~I_2:=\lbrace2\rbrace\rbrace$ of $\lbrace1,2,3\rbrace.$
In this case, we consider the set $\bsigma_{1,3}$ of sign choices $\sigma = \lbrace \sigma(I\times J) \rbrace\ci{I\times J \in \bfD}$, and the set $\bsigma_2$ of sign choices $\sigma = \lbrace \sigma(I) \rbrace\ci{I\in \bfD}.$
Given $\sigma_{1,3} \in \bsigma_{1,3}$ and $\sigma_2 \in \bsigma_2$ consider the martingale transform $T\ci{\sigma_{1,3}}^{1,3}$ acting on variables $1$ and $3$ and the martingale transform $T\ci{\sigma_2}^2$ acting on variable $2.$
The previous method can be adapted to show that
\begin{equation*}
\sup_{\sigma_{1,3}\in\bsigma_{1,3},~\sigma_2\in\Sigma_{2}}\Vert [T_{\sigma_{1,3}}^{1,3},[T_{\sigma_2}^{2},b]]\Vert\ci{L^p(\mu)\rightarrow L^{p}(\lambda)}\gtrsim\Vert b\Vert\ci{\text{bmo}^{\cI}\ci{\bfD}(\nu)},
\end{equation*}
while the corresponding upper bound was already proved by Airta (see \cite[Theorem~4.12]{airta}).
To this end, define the operator $\Xi_b$ by
\begin{equation*}
  \Xi_b f = \sum_{I,J,K\in\cD} (\xi\ci{I\times J \times K} b) f\ci{I\times J \times K} h\ci{I\times J \times K}, \qquad f \in L^2(\R^3),
\end{equation*}
where
\begin{equation*}
  \xi\ci{I\times J \times K} b = (b - \La b \Ra\ci{I\times K}^{1,3} - \La b \Ra\ci{J}^{2} + \La b \Ra\ci{I\times J \times K}^{1,2,3}) \1\ci{I\times J \times K}
\end{equation*}
and where we use $\La b \Ra\ci{I\times K}^{1,3}$ to denote the average taken on variables $1$ and $3$ taken over $I\times K$ (similarly for $\La b \Ra\ci{J}^{2}$).
Then one can repeat the arguments in Lemma \ref{l: symparaprod-Haarmult-p} to show that
\begin{equation*}
\sup_{\sigma_{1,3}\in\bsigma_{1,3},~\sigma_2\in\Sigma_{2}}\Vert [T_{\sigma_{1,3}}^{1,3},[T_{\sigma_2}^{2},b]]\Vert\ci{L^p(\mu)\rightarrow L^{p}(\lambda)}\gtrsim\Vert \Xi_b\Vert\ci{L^{p}(\mu) \rightarrow L^p(\lambda)}.
\end{equation*}
The lower bound for the operator norm of $\Xi_b$ follows a similar argument to that of $\Lambda_b,$ with the difference that in this case one needs to check the $\BMOprodD$ norm in variables $1,2$ and in variables $2,3.$
We focus on how to get the bound for variables $1$ and $2,$ as the other case is done in the same way.
Consider an arbitrary open set $\Omega \subseteq \R \times \R$ of non-zero finite measure and its characteristic function $\1\ci{\Omega}(x_1,x_2)$ in variables $1$ and $2.$
Fix $x_3$ and take $K \in \cD$ such that $x_3 \in K.$
Note that for $x \in \R\times\R\times\lbrace x_3 \rbrace$ we trivially have $\Xi_{b(x_{\overline{1,2}})} (\1\ci{\Omega} \otimes h\ci{K}) = (\Lambda_{b(x_{\overline{1,2}})}\1\ci{\Omega}) \otimes h\ci{K}.$
By testing the operator $\Xi_b$ on $\1\ci{\Omega} \otimes h\ci{K}$ one gets
\begin{align*}
  &\frac{1}{|K|^{1-p/2}} \Vert \Xi_b (\1\ci{\Omega} \otimes h\ci{K}) \Vert\ci{L^p(\lambda)}^p\\
  &=\int_{\R\times\R} \frac{1}{|K|} \int_K \left|\sum_{I,J \in \cD} (b-\La b \Ra\ci{I\times K}^{1,3}-\La b \Ra\ci{J}^{2}+\La b \Ra\ci{I\times J\times K}^{1,2,3})(\1\ci{\Omega})\ci{I\times J}\right|^p \lambda(x_1,x_2,x_3) dx\\
  &\leq \Vert \Xi_b\Vert\ci{L^{p}(\mu) \rightarrow L^p(\lambda)}^p \La \mu(\Omega,x_3) \Ra\ci{K}^{3}.
\end{align*}
Applying twice Lebesgue Differentiation Theorem and Fatou's Lemma we get
\begin{align*}
  &\Vert \Lambda_{b(x_{\overline{1,2}})}\1\ci{\Omega} \Vert\ci{L^p(\lambda(x_{\overline{1,2}}))}^{p}\\
  &=\int_{\R\times\R} \left|\sum_{I,J \in \cD} (b(x_{\overline{1,2}})-\La b(x_{\overline{1,2}}) \Ra\ci{I}^{1}-\La b(x_{\overline{1,2}}) \Ra\ci{J}^{2}+\La b(x_{\overline{1,2}}) \Ra\ci{I\times J}^{1,2})(\1\ci{\Omega})\ci{I\times J}\right|^p \lambda(x_{\overline{1,2}})\, dx_1\, dx_2\\
  &\leq \lim_{K \to x_3} \Vert \Xi_b\Vert\ci{L^{p}(\mu) \rightarrow L^p(\lambda)}^p \La \mu(\Omega,x_3) \Ra\ci{K}^{3}
  = \Vert \Xi_b\Vert\ci{L^{p}(\mu) \rightarrow L^p(\lambda)}^p \mu(\Omega,x_3),
\end{align*}
where $K \to x_3$ denotes that the limit is taken through a sequence of intervals containing $x_3$ with side length tending to $0,$ and where the last equality holds at almost every $x_3.$
Thus, by the bound \eqref{symparaprod-BMO-p} for operator $\Lambda_{b(x_{\overline{1,2}})}$ in terms of the dyadic one-weight biparameter product BMO norm, we get
\begin{equation*}
  \Vert \Xi_b\Vert\ci{L^{p}(\mu) \rightarrow L^p(\lambda)}
  \geq \esssup_{x_3} \Vert b(x_{\overline{1,2}}) \Vert\ci{\BMOprodD(\nu(x_{\overline{1,2}}))}.
\end{equation*}

The case of general commutators and indexed spaces can be worked out in a similar way, considering the appropriate multiparameter analogues of the operator $\Xi_b$, Lemma \ref{l: symparaprod-Haarmult-p} and equation \eqref{symparaprod-BMO-p}.

\end{document}